\documentclass[11pt]{amsart}
\usepackage{amscd,amsxtra,amssymb,mathrsfs, bbm}
\usepackage{stmaryrd}

\newtheorem{theorem}{Theorem}[section]
\newtheorem{corollary}[theorem]{Corollary}
\newtheorem{lemma}[theorem]{Lemma}
\newtheorem{proposition}[theorem]{Proposition}

\theoremstyle{definition}
\newtheorem{definition}[theorem]{Definition}
\newtheorem{remark}[theorem]{Remark}
\newtheorem{example}[theorem]{Example}

\DeclareMathAlphabet{\mathpzc}{OT1}{pzc}{m}{it}

\DeclareMathOperator{\Perf}{\mathsf{Perf}}
\DeclareMathOperator{\Com}{\mathsf{Com}}

\DeclareMathOperator{\Band}{\mathsf{Band}}
\DeclareMathOperator{\Sing}{\mathsf{Sing}}

\DeclareMathOperator{\coker}{\mathsf{coker}}

\renewcommand{\ker}{\mathsf{ker}}

\renewcommand{\dim}{\mathsf{dim}}

\DeclareMathOperator{\Coh}{\mathsf{Coh}}

\DeclareMathOperator{\MCM}{\mathsf{MCM}}

\DeclareMathOperator{\Hot}{\mathsf{Hot}}
\DeclareMathOperator{\Pro}{\mathsf{pro}}

\DeclareMathOperator{\Ob}{\mathsf{Ob}}

\DeclareMathOperator{\krdim}{\mathsf{kr.dim}}

\DeclareMathOperator{\gldim}{\mathsf{gl.dim}}

\DeclareMathOperator{\add}{\mathsf{add}}

\DeclareMathOperator{\Supp}{\mathsf{Supp}}

\DeclareMathOperator{\Hom}{\mathsf{Hom}}

\DeclareMathOperator{\Ext}{\mathsf{Ext}}

\DeclareMathOperator{\End}{\mathsf{End}}

\DeclareMathOperator{\Spec}{\mathsf{Spec}}

\input xy
\xyoption{all}

\newcommand{\kk}{\mathbbm{k}}
\newcommand{\bsm}{\begin{smallmatrix}}
\newcommand{\esm}{\end{smallmatrix}}

\renewcommand{\mod}{\mathsf{mod}}
\newcommand{\fdmod}{\mathsf{fdmod}}

\newcommand{\LL}{\mathbb{L}}

\newcommand{\kA}{\mathcal{A}}

\newcommand{\kC}{\mathcal{C}}
\newcommand{\kE}{\mathcal{E}}
\newcommand{\kF}{\mathcal{F}}

\newcommand{\kH}{\mathcal{H}}
\newcommand{\kI}{\mathcal{I}}

\newcommand{\kO}{\mathcal{O}}

\newcommand{\kP}{\mathcal{P}}
\newcommand{\kQ}{\mathcal{Q}}
\newcommand{\kK}{\mathcal{K}}
\newcommand{\kM}{\mathcal{M}}
\newcommand{\kN}{\mathcal{N}}

\newcommand{\kT}{\mathcal{T}}
\newcommand{\kS}{\mathcal{S}}
\newcommand{\kR}{\mathcal{R}}

\newcommand{\kU}{\mathcal{U}}

\newcommand{\cP}{\mathsf{P}}

\newcommand{\lar}{\longrightarrow}

\newcommand{\idm}{\mathfrak{m}}

\renewcommand{\AA}{\mathbb{A}}

\newcommand{\FF}{\mathbb{F}}
\newcommand{\GG}{\mathbb{G}}
\newcommand{\HH}{\mathbb{H}}

\newcommand{\TT}{\mathbb{T}}
\newcommand{\II}{\mathbb{I}}
\newcommand{\EE}{\mathbb{E}}

\newcommand{\PP}{\mathbb{P}}
\newcommand{\UU}{\mathbb{U}}
\newcommand{\ZZ}{\mathbb{Z}}
\newcommand{\XX}{\mathbb{X}}
\newcommand{\YY}{\mathbb{Y}}

\newcommand{\String}[3]{\mathcal{S}_{#1}(#2)[#3]}
\newcommand{\lo}{\gamma}
\newcommand{\ro}{\alpha}
\newcommand{\ru}{\delta}
\newcommand{\lu}{\beta}
\newcommand{\li}{-}
\newcommand{\re}{+}

\usepackage{hyperref}

\begin{document}

\title[Singularity category of a non-commutative
resolution]{Relative singularity category of a non-commutative
resolution of singularities}

\author{Igor Burban}
\address{
Universit\"at zu K\"oln,
Mathematisches Institut,
Weyertal 86-90,
D-50931 K\"oln,
Germany
}
\email{burban@math.uni-koeln.de}

\author{Martin Kalck}
\address{
Mathematisches Institut,
Universit\"at Bonn,
Endenicher Allee 60,
D-53115 Bonn,
Germany
}

\email{kalck@math.uni-bonn.de}


\begin{abstract}
In this article, we study a
 triangulated category associated with  a non-commutative resolution of singularities.
In particular, we give a complete description of this category
in the case of  a curve with nodal singularities, classifying its  indecomposable objects and computing its Auslander--Reiten quiver and K--group.
\end{abstract}

\maketitle

\section{Introduction}

This article grew up from an attempt to generalize the following statement, which is a consequence of
a theorem of   Buchweitz \cite[Theorem 4.4.1]{Buchweitz} and results on idempotent completions of triangulated categories
\cite{KellerVandenBergh, Orlov}. Let $X$ be an algebraic  variety with    \emph{isolated
Gorenstein} singularities, $Z = \Sing(X) = \bigl\{x_1, \dots, x_p\bigr\}$ and
$\widehat{O}_i := \widehat{\kO}_{X, x_i}$ for all $1 \le i \le p$. Then we have an equivalence
of triangulated categories
\begin{equation}\label{E:Buchweitz}
\left(\frac{D^b\bigl(\Coh(X)\bigr)}{\Perf(X)}\right)^\omega \stackrel{\sim}\lar  \bigvee_{i = 1}^p
\underline{\MCM}\!\left(\widehat{O}_i\right).
\end{equation}
The left-hand side of (\ref{E:Buchweitz}) stands for  the idempotent completion of the Verdier quotient
${D^b\bigl(\Coh(X)\bigr)}/{\Perf(X)}$ (known to be triangulated by
\cite{BalmerSchlichting}),  whereas on the right-hand side $\underline{\MCM}(\widehat{O}_i)$ denotes
the stable
category of  maximal Cohen-Macaulay modules over  $\widehat{O}_i$.

We want to generalize this construction as follows. Let $\kF' \in \Coh(X)$,
$\kF := \kO \oplus \kF'$ and $\kA := {\mathcal End}_X(\kF)$. Consider the ringed space
$\mathbb{X} := (X, \kA)$. It is  well-known that  the functor
$$
\kF \stackrel{\mathbb{L}}\otimes_X \,-\, : \Perf(X) \lar D^b\bigl(\Coh(\XX)\bigr)
$$
is fully faithful, see for instance \cite[Theorem 2]{Tilting}. If $\mathrm{gl.dim}\bigl(\Coh(\XX)\bigr) < \infty$ then $\XX$ can be viewed
as a \emph{non-commutative} (or \emph{categorical}) resolution of singularities of $X$, in
the spirit of works of Van den Bergh \cite{vdBergh2}, Kuznetsov \cite{Kuznetsov} and Lunts
\cite{Lunts}. To measure the difference between $\Perf(X)$ and $D^b\bigl(\Coh(\XX)\bigr)$, we
suggest to study the triangulated category
\begin{equation}
\Delta_{X}(\XX) := \left(\frac{D^b\bigl(\Coh(\XX)\bigr)}{\Perf(X)}\right)^\omega,
\end{equation}
which we shall call \emph{relative singularity category}.
Assuming $\kF$ to be locally free on $U := X \setminus Z$, we prove an analogue of the
``localization equivalence'' (\ref{E:Buchweitz}) for the category $\Delta_{X}(\XX)$.
Using the negative K-theory of derived categories  of Schlichting \cite{Schlichting},
we also describe the Grothendieck group of $\Delta_{X}(\XX)$.

The main result of our  article is a complete description of $\Delta_{Y}(\YY)$ in case
$Y$ is an arbitrary  \emph{curve} with \emph{nodal singularities} and $\kF' := \kI_Z$ is the ideal sheaf
 of the singular locus of $Y$. We show that $\Delta_{Y}(\YY)$
splits into a union of $p$ blocks: $\Delta_{Y}(\YY) \stackrel{\sim}\lar  \bigvee _{i = 1}^p \Delta_i$, where
$p$ is the number of singular points of $Y$.  Moreover,
each block $\Delta_i$  turns out to be  equivalent to the category $\Delta_{\mathsf{nd}}$
defined as follows:
$$
\Delta_{\mathsf{nd}}:= \frac{\Hot^b\bigl(\Pro(A_{\mathsf{nd}})\bigr)}{\Hot^b\bigl(\add(P_*)\bigr)},
$$
where $A_{\mathsf{nd}}$ is the completed  path algebra of the following quiver with relations:
$$
\begin{xy}\SelectTips{cm}{}
\xymatrix
{
\li \ar@/^/[rr]^{ \ro }  & &  \ast \ar@/^/[ll]^{ \lu }
 \ar@/_/[rr]_{ \ru }
 & &
\ar@/_/[ll]_{ \lo } \re}\end{xy}  \qquad  \ru   \ro  = 0, \quad   \lu   \lo  = 0
$$
and $P_{*}$ is the indecomposable projective $A_{\mathsf{nd}}$--module corresponding to the vertex $*$.
We prove that the category $\Delta_{\mathsf{nd}}$ is idempotent complete and $\Hom$--finite. Moreover, we give
a complete classification of  indecomposable objects of $\Delta_{\mathsf{nd}}$.

\medskip
\noindent
Finally, we show that $\Delta_{\mathsf{nd}}$ has the following interesting  description:
$$
\Delta_{\mathsf{nd}} \stackrel{\sim}\lar  \left(\frac{D^b\bigl(\Lambda-\mod\bigr)}{\Band(\Lambda)}\right)^\omega,
$$
where $\Lambda$ is the path algebra of the following quiver with relations
$$
\begin{xy}\SelectTips{cm}{}
\xymatrix
{
\circ \ar@/^/[rr]^{a} \ar@/_/[rr]_{c} & & \circ \ar@/^/[rr]^{b} \ar@/_/[rr]_{d} & & \circ
}
\end{xy}
\qquad ba = 0,  \quad dc =  0
$$
and $\Band(\Lambda)$ is the category of the \emph{band objects} in $D^b\bigl(\Lambda-\mod\bigr)$, i.e.~those objects which are invariant under
the Auslander--Reiten translation in $D^b\bigl(\Lambda-\mod\bigr)$. Using this result, we describe
the Auslander--Reiten quiver of $\Delta_{\mathsf{nd}}.$

\medskip
\noindent
\emph{Acknowledgement}. We would like to thank Nicolas Haupt, Jens Hornbostel,
Bernhard Keller, Henning Krause, Helmut Lenzing, Michel Van den Bergh and Dong Yang for helpful discussions of parts of this article. We are  also grateful to
the anonymous referee for his/her  comments and suggestions. This  work  was supported  by the DFG grant Bu--1866/2--1.

\section{Generalities on the non-commutative category of singularities}
Let $\kk$ be an algebraically closed field, $X$ be  a separated  excellent Noetherian  scheme over $\kk$ such that any
coherent sheaf on $X$ is a quotient of a locally free sheaf, and  $Z$ be
 the singular locus of $X$. Let $\kF'$ be a coherent sheaf on $X$, $\kF = \kO \oplus \kF'$ and $\kA := {\mathcal End}_X(\kF)$. Consider the non-commutative ringed space $\XX = (X, \kA)$.
 Note that $\kF$ is a \emph{locally projective} coherent left $\kA$--module.
  The following result is well-known, see e.g.~\cite[Theorem 2]{Tilting}.
\begin{proposition}\label{P:startingpoint}
The functor $
\FF := \kF \stackrel{\LL}\otimes_X \,-\,\colon \Perf(X) \rightarrow  D^b\bigl(\Coh(\XX)\bigr)
$
is fully faithful.
\end{proposition}

\noindent
Let $\cP(X)$ be the essential image of $\Perf(X)$ under $\FF$.
This category can be characterized in the following intrinsic way, see for instance
\cite[Proposition 2]{Tilting}\footnote{Although the quoted results
  were stated  in \cite{Tilting} in a weaker form, their  proofs  can be generalized literally to our case.}
$$
\Ob\bigl(\cP(X)\bigr) = \Bigl\{\kH^\bullet \in
\Ob\Bigl(D^b\bigl(\Coh(\XX)\bigr)\Bigr) \, \Big| \, \kH_x^\bullet \in
\mathrm{Im}\Bigl(\Hot^b\bigl(\add(\kF_x)\bigr) \lar  D^b\bigl(\kA_x-\mod\bigr)\Bigr)\Bigr\}.
$$
\begin{definition}\label{D:DefCatSing} In the above notations, the \emph{relative singularity category}
$\Delta_{X}(\XX)$ is the idempotent completion of the Verdier quotient $D^b\bigl(\Coh(\XX)\bigr)/\cP(X)$.
Recall that according to  Balmer and Schlichting \cite{BalmerSchlichting},
$\Delta_{X}(\XX)$ has a natural structure of  a triangulated category.
\end{definition}

\begin{remark}
In case $X$ is an affine scheme,   triangulated categories of the form  $\Delta_{X}(\XX)$
were also considered by Chen \cite{Chen2},
 Thanhoffer de V\"olcsey and Van den Bergh \cite{ThanhofferVandenBergh}.
\end{remark}

\noindent
The following result seems to be well-known to experts. However, we were not able to find
a reference in the literature and therefore give a proof here.

\begin{lemma}\label{P:subtleProp}
Let $A$ be a ring and $O$ be its center. Assume that $O$ is Noetherian of Krull dimension $d$ and
$A$  is finitely generated as a left $O$--module. For any $0 \le e \le d$ let
$A-\mod^e$ be the full subcategory of left Noetherian $A$--modules $A-\mod$, whose support over $O$ is at most $e$--dimensional and $D^b_e(A-\mod)$ be the full subcategory of $D^b(A-\mod)$ consisting of complexes whose cohomology belongs to $A-\mod^e$.  Then the canonical functor
$$
D^b\left(A-\mod^e\right) \lar D^b_e(A-\mod)
$$
is an equivalence of triangulated categories.\footnote{We would like to thank B.~Keller and M.~Van den Bergh for an enlightening discussion on this subject.}
\end{lemma}

\begin{proof}
By  \cite[Proposition 1.7.11]{KashiwaraSchapira},
it is sufficient to show the following

\vspace{1mm}
\noindent
\underline{Statement}.
Let $M$ be an arbitrary object of $A-\mod^e$, $N$ an arbitrary object of $A-\mod$ and $
\phi\colon M \rightarrow N$  an arbitrary injective $A$--linear map. Then there exists an object
$K$ of $A-\mod^e$ and a morphism $\psi\colon N \rightarrow K$ such that $\psi \phi$ is injective.

\vspace{1mm}
\noindent
Indeed, let $E = E(M)$ be an injective envelope of $M$ and $\theta\colon M \rightarrow E$ the
corresponding embedding. Then there exists a morphism $\alpha\colon N \rightarrow E$ such that
$\alpha \phi = \theta$. Note that $K := \mathsf{Im}(\alpha)$ is a left Noetherian $A$--module.
As in \cite[Lemma 3.2.5]{BrunsHerzog} one can  show that for any $\mathfrak{p} \in \Spec(O)$ we
have: $E_\mathfrak{p} \cong E(M_\mathfrak{p})$. Hence, $K_{\mathfrak{p}}=0$ for all $\mathfrak{p} \notin \Supp(M)$. This implies that $\krdim\bigl(\Supp(K)\bigr) \le e$.
\end{proof}

\begin{remark}
 Lemma \ref{P:subtleProp} is no longer true if  $A$ is assumed to be just left Noetherian. For example, let $\mathfrak{g}$
be a finite dimensional simple Lie algebra  over $\mathbb{C}$ and $U = U(\mathfrak{g})$ its universal enveloping algebra. Then $U$ is left Noetherian, see for instance \cite[Section I.7]{McConnellRobson}.
 By Weyl's complete reducibility theorem, the category $U-\mod^0$ of finite dimensional
left $U$--modules is semi-simple. However, higher extensions between finite dimensional modules
do not necessarily vanish in $U-\mod$, see for instance \cite{Humphreys}.
 In particular, the triangulated categories
$D^b(U-\mod^0)$ and $D^b_0(U-\mod)$ are \emph{not} equivalent.
\end{remark}

\noindent
Globalizing the proof of Lemma  \ref{P:subtleProp}, we get the following result.

\begin{lemma}\label{P:DercCatandSupp}
Let $\Coh_Z(\XX)$ be the category of coherent left $\kA$--modules, whose  support belongs to  $Z$
and $D^b_Z\bigl(\Coh(\XX)\bigr)$ be the full subcategory of $D^b\bigl(\Coh(\XX)\bigr)$ consisting
of complexes whose cohomology is supported at $Z$. Then the canonical functor
$$
D^b\bigl(\Coh_Z(\XX)\bigr) \lar D^b_Z\bigl(\Coh(\XX)\bigr)
$$
is an equivalence of triangulated categories.
\end{lemma}

\noindent
Our next goal is to prove that the category $\Delta_{X}(\XX)$ depends only on an open neighborhood
of the singular locus $Z$.

\begin{proposition}\label{P:FullyFaithf}
Let $D^b_Z\bigl(\Coh(\XX)\bigr)$ be the full subcategory of $D^b\bigl(\Coh(\XX)\bigr)$
consisting of complexes whose cohomology is supported in $Z$ and
$\cP_Z(X) = \cP(X) \cap  D^b_Z\bigl(\Coh(\XX)\bigr)$. Then the canonical functor
$$
\HH\colon \quad
\frac{D^b_Z\bigl(\Coh(\XX)\bigr)}{\cP_Z(X)} \lar
\frac{D^b\bigl(\Coh(\XX)\bigr)}{\cP(X)}
$$
is fully faithful.
\end{proposition}

\begin{proof} Our approach is inspired by a recent paper of Orlov \cite{Orlov}.  By \cite[Proposition 1.6.10]{KashiwaraSchapira}, it is sufficient to show that for any $\kP^\bullet \in \Ob\bigl(\cP(X)\bigr)$,
$\kC^\bullet \in \Ob\bigl(D^b_Z\bigl(\Coh(\XX)\bigr)\bigr)$ and $\varphi \colon \kP^\bullet \rightarrow \kC^\bullet$ there
exists $\kQ^\bullet \in \Ob\bigl(\cP_Z(X)\bigr)$ and a factorization
$$
\begin{xy}\SelectTips{cm}{}
\xymatrix{
\kP^\bullet \ar[rr]^{\varphi} \ar[dr]_{\varphi'} & & \kC^\bullet \\
& \kQ^\bullet \ar[ur]_{\varphi''}&
}
\end{xy}
$$
By Lemma \ref{P:DercCatandSupp}, we know that the functor $D^b_Z\bigl(\Coh(\XX)\bigr)
\rightarrow D^b\bigl(\Coh_Z(\XX)\bigr)$ is an equivalence of categories. Hence, we may
without loss of generality assume that $\kC^\bullet$ is a bounded complex of objects of $\Coh_Z(\XX)$.
Let $\kI=\kI_{Z}$ be the ideal sheaf of  $Z$. Then there exists $t \ge 1$ such that
$\kI^t$ annihilates every term of  $\kC^\bullet$. Consider the ringed space
$\YY = \bigl(Z, \kA/\kI^t\bigr)$. Then we have a morphism of ringed spaces $\eta \colon \YY \rightarrow \XX$ and
an adjoint pair
$$
\left\{
\begin{array}{ll}
\eta_* = \mathsf{forgetf}\colon             & D^-\bigl(\Coh(\YY)\bigr) \rightarrow D^-\bigl(\Coh(\XX)\bigr) \\
\eta^* = \kA/\kI^t \stackrel{\LL}\otimes_\kA \,-\, \colon & D^-\bigl(\Coh(\XX)\bigr) \rightarrow D^-\bigl(\Coh(\YY)\bigr).
\end{array}
\right.
$$
Next, there exists $\kE^\bullet \in \Ob\bigl(D^b\bigl(\Coh(\YY)\bigr)\bigr)$ such that
$\kC^\bullet = \eta_*(\kE^\bullet)$. Moreover, we have an isomorphism
$
\gamma \colon \Hom_{\YY}\bigl(\eta^*\kP^\bullet, \kE^\bullet\bigr) \lar
\Hom_{\XX}\bigl(\kP^\bullet, \eta_*(\kE^\bullet)\bigr)
$
such that for $\psi \in \Hom_{\YY}\bigl(\eta^*\kP^\bullet, \kE^\bullet\bigr)$ the corresponding
morphism $\varphi = \gamma(\psi)$ fits into the commutative diagram
$$
\begin{xy}\SelectTips{cm}{}
\xymatrix{
\kP^\bullet \ar[rr]^{\xi_{\kP^\bullet}} \ar[rd]_{\varphi} & & \eta_* \eta^* \kP^\bullet  \ar[ld]^{\eta_*(\psi)} \\
& \eta_* \kE^\bullet
}
\end{xy}
$$
where $\xi \colon \mathbbm{1}_{D^-(\XX)} \rightarrow \eta_* \eta^*$ is the unit of adjunction.
Thus,  it is sufficient to find a factorization of the morphism $\xi_{\kP^\bullet}$ through
an object of $\cP_Z(X)$.

By definition of $\cP(X)$, there exists a bounded complex of locally free
$\kO_X$--modules $\kR^\bullet$ such that the complexes $\kP^\bullet$ and $\kF \otimes_X \kR^\bullet$ are isomorphic in
$D^b\bigl(\Coh(\XX)\bigr)$. Note
that we have the following commutative diagram in the category $\mathsf{Com}^b(\XX)$
of bounded complexes of coherent left $\kA$--modules:
$$
\begin{xy}\SelectTips{cm}{}
\xymatrix
{
\kF \otimes_X \kR^\bullet \ar[rr]^{\mathbbm{1} \otimes \theta_{\kR^\bullet}} \ar[rd]_{\zeta_{\kR^\bullet}} & & \kF \otimes_X \bigl(\kO/\kI^t \otimes_X \kR^\bullet\bigr) \ar[ld]^{\cong} \\
& \kA/\kI^t \otimes_\kA \bigl(\kF \otimes_X \kR^\bullet\bigr) &
}
\end{xy}
$$
where $\zeta_{\kR^\bullet} = \xi_{\kP^\bullet}$ in $D^-\bigl(\Coh(\XX)\bigr)$ and
$\theta_{\kR^\bullet} \colon \kR^\bullet \rightarrow \kO/\kI^t \otimes_X \kR^\bullet$ is the canonical map.
Since any coherent sheaf on $X$ is a quotient of a locally free sheaf,
there exists a bounded complex $\kK^\bullet$
of locally free $\kO_X$--modules (Koszul complex of $\kI^t$)
$$
\kK^\bullet = \bigl(
0 \lar \kK^m \lar \dots \lar \kK^1 \lar \kK^0 \lar 0
\bigr)
$$
such that
\begin{itemize}
\item $\kK^0 = \kO$ and  $\kH^0(\kK^\bullet) \cong  \kO/\kI^t$,
\item $\kH^{-i}(\kK^\bullet)$ are supported at $Z$ for all $1 \le i \le m$.
\end{itemize}
Hence, we have a factorization of the canonical morphism
$\kO \rightarrow \kO/\kI^t$ in the category of complexes  $\mathsf{Com}^b\bigl(\Coh(X)\bigr)$:
$\kO[0] \rightarrow \kK^\bullet \rightarrow \kO/\kI^t[0]$, which induces  a factorization
$$
\kR^\bullet \lar \kK^\bullet \otimes_X \kR^\bullet \lar \kO/\kI^t \otimes_X \kR^\bullet
$$
of the canonical map $\theta_{\kR^\bullet}$. Note that the complex
$\kK^\bullet \otimes_X \kR^\bullet$ is perfect and its cohomology is supported at $Z$. Hence, we get
a factorization of the (derived) adjunction unit  $\xi_{\kP^\bullet}$
$$
\kP^\bullet \cong \kF  \stackrel{\LL}\otimes_X \kR^\bullet \lar \kQ^\bullet := \kF \stackrel{\LL}\otimes_X \bigl(\kK^\bullet \stackrel{\LL}\otimes_X \kR^\bullet\bigr)
 \lar \kA/\kI^t \stackrel{\LL}\otimes_\kA \bigl(\kF \stackrel{\LL}\otimes_X \kR^\bullet\bigr) \cong \kA/\kI^t \stackrel{\LL}\otimes_\kA \kP^\bullet
$$
 we are looking for. This concludes the proof.
\end{proof}

\begin{theorem}\label{T:keylocaliz}
In the notations of Proposition \ref{P:FullyFaithf}, the induced functor
$$
\HH^\omega\colon \quad
\left(
\frac{D^b_Z\bigl(\Coh(\XX)\bigr)}{\cP_Z(X)}
\right)^\omega
\lar
\left(
\frac{D^b\bigl(\Coh(\XX)\bigr)}{\cP(X)}
\right)^\omega
$$
is an equivalence of triangulated categories.
\end{theorem}

\begin{proof}
Proposition \ref{P:FullyFaithf} implies that the functor $\HH^\omega$ is fully faithful. Hence,
we have to show it is essentially surjective. It suffices to prove the following

\vspace{1mm}
\noindent
\underline{Statement}. For any
$ \kM^\bullet \in \Ob\left({D^b\bigl(\Coh(\XX)\bigr)}/{\cP(X)}
\right)$ there exist  $ \widetilde\kM^\bullet \in
\Ob\left({D^b\bigl(\Coh(\XX)\bigr)}/{\cP(X)}
\right)$  and $
\kN^\bullet \in \Ob\left({D^b_Z\bigl(\Coh(\XX)\bigr)}/{\cP_Z(X)}
\right)$ such that  $\kM^\bullet \oplus \widetilde\kM^\bullet \cong \HH(\kN^\bullet).
$

\vspace{1mm}
\noindent
Note that we have the following diagram of categories and functors
$$
\begin{xy}\SelectTips{cm}{}
\xymatrix{
\frac{\displaystyle D^b\bigl(\Coh(\XX)\bigr)}{\displaystyle D^b_Z\bigl(\Coh(\XX)\bigr)}
\ar[rr]^{\AA} & &
D^b\left(\frac{\displaystyle \Coh(\XX)}{\displaystyle \Coh_Z(\XX)}\right) \ar[dd]^{\imath^*} \\
D^b\bigl(\Coh(\XX)\bigr) \ar[u]^{\PP} & & \\
\Perf(X) \ar[u]^{\kF \stackrel{\mathbbm{L}}\otimes_X \,-\,} \ar[r]^{\jmath^*} & \Perf(U)
\ar[r]^{\kF\big|_{U} \stackrel{\mathbbm{L}}\otimes_U \,-\,}  & D^b\bigl(\Coh(\mathbb{U})\bigr), \\
}
\end{xy}
$$
where both compositions $\Perf(X) \rightarrow D^b\bigl(\Coh(\mathbb{U})\bigr)$ are isomorphic.
\begin{itemize}
\item The functor $\PP$ is the canonical projection on the Verdier quotient.
\item The functor $\AA$ is the canonical equivalence of triangulated categories constructed by Miyachi \cite[Theorem 3.2]{Miyachi}.
\item For  $U = X \setminus Z$ let  $\UU$ be  the ringed space  $\bigl(U, \, \kA\big|_U)$. The functor
$\jmath^*$ is  the canonical restrictions on an open subset. The functor $\imath^*$ is an equivalence
of triangulated categories induced by a canonical  equivalence of abelian categories
$\Coh(\XX)/\Coh_Z(\XX) \rightarrow \Coh(\UU)$.
\item Since the coherent sheaf $\kF\big|_U$ is locally free, the functor
$$\kF\big|_U \stackrel{\mathbbm{L}}\otimes_U \,-\,\colon \quad
\Perf(U) = D^b\bigl(\Coh(U)\bigr) \lar  D^b\bigl(\Coh(\UU)\bigr)$$
 is an equivalence of triangulated categories induced by a Morita-type equivalence $\kF\big|_U \otimes_U \,-\,\colon
\Coh(U) \rightarrow \Coh(\UU)$.
\end{itemize}
By a result of Thomason and Trobaugh \cite[Lemma 5.5.1]{ThomasonTrobaugh}, for any $\kS^\bullet
\in \Ob\bigl(\Perf(U)\bigr)$ there exist $\widetilde\kS^\bullet \in \Ob\bigl(\Perf(U)\bigr)$
and $\kR^\bullet \in \Ob\bigl(\Perf(X)\bigr)$ such that
$
\jmath^* \kR^\bullet \cong \kS^\bullet \oplus \widetilde\kS^\bullet.
$
Using the fact that $\AA$,  $\imath^*$ and $\kF\big|_U \stackrel{\mathbbm{L}}\otimes_U \,-\,$ are
equivalences of categories, this implies that for any $\kM^\bullet \in
\Ob\bigl(D^b\bigl(\Coh(\XX)\bigr)\bigr)$ there exist $\widetilde\kM^\bullet \in  \Ob\bigl(D^b\bigl(\Coh(\XX)\bigr)\bigr)$ and $\kR^\bullet \in \Ob\bigl(\Perf(X)\bigr)$
such that $\kP^\bullet:= \kF \stackrel{\mathbbm{L}}\otimes_X \kR^\bullet$
is isomorphic to $\kM^\bullet \oplus \widetilde\kM^\bullet$ in the Verdier quotient
$D^b\bigl(\Coh(\XX)\bigr)/D^b_Z\bigl(\Coh(\XX)\bigr)$. The last statement is equivalent to the fact
that there exists  $\kT^\bullet \in  \Ob\bigl(D^b\bigl(\Coh(\XX)\bigr)\bigr)$
and a pair of distinguished triangles
$$
\kC_\xi^\bullet  \lar \kT^\bullet \stackrel{\xi}\lar  \kM^\bullet \oplus \widetilde\kM^\bullet \lar
\kC_\xi^\bullet[1] \quad
\mathrm{and}
\quad
\kC_\theta^\bullet  \lar \kT^\bullet \stackrel{\theta}\lar \kP^\bullet \lar \kC_\theta^\bullet[1]
$$
in $D^b\bigl(\Coh(\XX)\bigr)$ such that
$\kC_\xi^\bullet$ and $\kC_\theta^\bullet$   belong to the category $D^b_Z\bigl(\Coh(\XX)\bigr)$.
Since $\kC_\theta^\bullet$ and $\kT^\bullet$ are isomorphic in the Verdier quotient
$D^b\bigl(\Coh(\XX)\bigr)/\cP(X)$, we get a distinguished triangle
$$
\kC_\xi^\bullet \stackrel{\alpha}\lar \kC_\theta^\bullet \lar \kM^\bullet \oplus \widetilde\kM^\bullet \lar
\kC_\xi^\bullet[1]
$$
in  $D^b\bigl(\Coh(\XX)\bigr)/\cP(X)$.
The functor $\HH\colon \, D^b_Z\bigl(\Coh(\XX)\bigr)/\cP_Z(X) \rightarrow
D^b\bigl(\Coh(\XX)\bigr)/\cP(X)$ is fully faithful, see  Proposition \ref{P:FullyFaithf}. Hence,
 $\kM^\bullet \oplus \widetilde\kM^\bullet$
belongs to the essential image of $\HH$.
Thus, the  functor  $\HH^\omega$ is essentially surjective, what concludes the proof.
\end{proof}

\noindent
From now on, we assume $X$ has only \emph{isolated singularities} and
$Z = \Sing(X) = \bigl\{x_1, \dots, x_p \bigr\}$. For any $1 \le i \le p$ we denote
$O_i := \kO_{x_i}$, $\mathfrak{m}_i$ the maximal ideal in $O_i$, $A_i = \kA_{x_i}$
and $F_i = \kF_{x_i}$. Next, we set $\widehat{O}_i = \varprojlim O_i/ \idm_i^t O_i$ to be the $\mathfrak{m}_i$--adic completion
of $O_i$, $\widehat{A}_i:= \varprojlim A_i/ \idm_i^t A_i$ and $\widehat{F}_i:= \varprojlim F_i/ \idm_i^t F_i$. Note that $\widehat{A}_i \cong \End_{\widehat{O}_i}(\widehat{F}_i)$.
Let $A_i-\fdmod$ denote the category of finite dimensional left $A_i$--modules. In this case,
Lemma \ref{P:DercCatandSupp} yields the following statement.

\begin{lemma}\label{L:splittingFinSupp}
The canonical functor
$
\vee_{i = 1}^p D^b\bigl(A_i-\fdmod\bigr) \rightarrow D^b_Z\bigl(\Coh(\XX)\bigr)
$
is an equivalence of triangulated categories. Let $\cP_i$ be the full subcategory
of $D^b\bigl(A_i-\fdmod\bigr)$ consisting of objects admitting a bounded resolution
by objects of $\add(F_i)$. Then this functor restricts to an equivalence
$\vee_{i = 1}^p \cP_i \rightarrow \cP_Z(X)$.
\end{lemma}

\noindent
Our next aim is to  show that the Verdier quotient $D^b\bigl(A_i-\fdmod\bigr)/\cP_i$ does not change
under passing to the completion.

\begin{lemma}
Let $\Perf_{\mathsf{fd}}(O_i)$ (respectively $\Perf_{\mathsf{fd}}(\widehat{O}_i)$) be the full
subcategory of $D^b\bigl(O_i-\fdmod)$ (respectively $D^b\bigl(\widehat{O}_i-\fdmod)$) consisting
of those complexes which are quasi-isomorphic to a bounded complex of finite rank free
${O}_i$--~(respectively $\widehat{O}_i$)--modules.
Let $\Perf_{\mathsf{fd}}(O_i) \rightarrow \Perf_{\mathsf{fd}}(\widehat{O}_i)$ and
$D^b\bigl(A_i-\fdmod\bigr) \rightarrow D^b\bigl(\widehat{A}_i-\fdmod\bigr)$ be the exact functors
induced by taking the completion. Then they are both equivalences of categories. Moreover,
we have a diagram of categories and functors
$$
\begin{xy}\SelectTips{cm}{}
\xymatrix{
\Perf_{\mathsf{fd}}(O_i) \ar[rr] \ar[d]_{F_i \stackrel{\LL}\otimes_{O_i} \,-\,} & & \Perf_{\mathsf{fd}}(\widehat{O}_i)
\ar[d]^{\widehat{F}_i \stackrel{\LL}\otimes_{\widehat{O}_i} \,-\,} \\
D^b\bigl(A_i-\fdmod\bigr) \ar[rr] & & D^b\bigl(\widehat{A}_i-\fdmod\bigr),
}
\end{xy}
$$
where both compositions $\Perf_{\mathsf{fd}}(O_i) \rightarrow D^b\bigl(\widehat{A}_i-\fdmod\bigr)$
are isomorphic.
\end{lemma}

\begin{proof}
Since the functors $O_i-\fdmod \rightarrow \widehat{O}_i-\fdmod$ and
$A_i-\fdmod \rightarrow \widehat{A}_i-\fdmod$ are equivalences of categories, they
induce equivalences $D^b\bigl(O_i-\fdmod\bigr) \rightarrow D^b\bigl(\widehat{O}_i-\fdmod\bigr)$ and $D^b\bigl(A_i-\fdmod\bigr) \rightarrow D^b\bigl(\widehat{A}_i-\fdmod\bigr)$. In particular,
the functor $\Perf_{\mathsf{fd}}(O_i) \rightarrow \Perf_{\mathsf{fd}}(\widehat{O}_i)$ is fully faithful.
In order to show it is essentially surjective, it is sufficient to prove that a non-perfect
complex can not become perfect after applying the completion functor.
Indeed, $X^\bullet \in \Ob\bigl(D^b\bigl(O_i-\mod\bigr)\bigr)$ is perfect if and only if there
exists $n_0 \in \mathbb{N}$ such that for all $n \ge n_0$ we have:
$\Hom\bigl(X^\bullet, O_i/\idm_i[n]\bigr) = 0$. But this property is obviously preserved under the passing to the completion.
\end{proof}

\begin{corollary}\label{C:Localization}
Let $\widetilde{\cP}_i$ (respectively $\widehat{\cP}_i$) be the essential image of
the triangle functor $\widehat{F}_i \stackrel{\LL}\otimes_{\widehat{O}_i} \,-\,\colon \Perf(\widehat{O}_i) \rightarrow
D^b\bigl(\widehat{A}_i-\mod\bigr)$ (respectively $\widehat{F}_i \stackrel{\LL}\otimes_{\widehat{O}_i} \,-\,\colon
\Perf_{\mathsf{fd}}(\widehat{O}_i) \rightarrow
D^b\bigl(\widehat{A}_i-\fdmod\bigr)$). Then we have an equivalence of triangulated categories
$$
\frac{D^b\bigl({A}_i-\fdmod\bigr)}{\cP_i} \lar
\frac{D^b\bigl(\widehat{A}_i-\fdmod\bigr)}{\widehat{\cP}_i}.
$$
Going along the same lines as in Theorem \ref{T:keylocaliz},  one can show
that the canonical functor
$$
\left(
\frac{D^b\bigl(\widehat{A}_i-\fdmod\bigr)}{\widehat{\cP}_i}
\right)^\omega
\lar
\left(
\frac{D^b\bigl(\widehat{A}_i-\mod\bigr)}{\widetilde{\cP}_i}
\right)^\omega
$$
is an equivalence. Summing up, we have equivalences of triangulated categories
$$
\bigvee_{i=1}^p \left(
\frac{D^b\bigl(\widehat{A}_i-\mod\bigr)}{\widetilde{\cP}_i}
\right)^\omega \stackrel{\sim}\longleftarrow
\bigvee_{i=1}^p \left(
\frac{D^b\bigl({A}_i-\fdmod\bigr)}{{\cP}_i}
\right)^\omega  \stackrel{\sim}\lar
\left(
\frac{D^b\bigl(\Coh(\XX)\bigr)}{\cP(X)}
\right)^\omega =: \Delta_{X}(\XX).
$$
\end{corollary}

Now, let  $Y$ be a nodal algebraic curve, $Z = \{x_1, \dots, x_p\}$ the singular locus of $Y$,
$\kI = \kI_Z$ the ideal sheaf of $Z$ and  $\kF = \kO \oplus \kI$. Then
$\kA = {\mathcal End}_Y(\kF)$ is the \emph{Auslander sheaf of orders} introduced in \cite{Tilting}.
Let $\mathbb{Y} = (Y, \kA)$ be the corresponding non-commutative curve.
According to \cite[Theorem 2]{Tilting}, we have:  $\mathsf{gl.dim}\bigl(\Coh(\YY)\bigr) = 2$. Thus,
$\YY$ is a non-commutative resolution of $Y$ and  Corollary \ref{C:Localization} specializes to the following statement.

\begin{corollary}\label{C:nodalcurve}
The triangulated category
$\Delta_Y(\YY)$ splits into a union of $p$ blocks $\Delta_{\mathsf{nd}}$, where $\Delta_{\mathsf{nd}}$
is the ``local'' contribution of a singular point of $Y$ (see also Section \ref{S:ReprTheory} for
an explicit  description  of $\Delta_{\mathsf{nd}}$).
\end{corollary}

\noindent
The goal of  the subsequent part of this article is to answer the following questions.
\begin{itemize}
\item Is the category $\Delta_Y(\YY)$ $\Hom$--finite? What are its indecomposable objects?
\item What is the Grothendieck group of $\Delta_Y(\YY)$?
\item Assume $E$ is a plane nodal cubic curve. What is the relation of
$\Delta_E(\EE)$ with the ``quiver description'' of $D^b\bigl(\Coh(\EE)\bigr)$ from
\cite[Section 7]{Tilting}?
\end{itemize}

\section{On the K-theory of the relative singularity category $\Delta_{X}(\XX)$ }\label{S:KTheory}
Let $O$ be a \emph{complete Gorenstein} local ring and  $F =
O \oplus F_1 \oplus \dots \oplus F_r \in O-\mod$, where $F_1, \dots, F_r$ are
indecomposable and pairwise non-isomorphic and such that
$O$ does not belong to $\add(F_1 \oplus \dots \oplus F_r)$. Let $A=\End_{O}(F)$, $\cP(O)$ be the essential image
of $\Perf(O)$ under the exact embedding $F \stackrel{\LL}\otimes_O \,-\, \colon
\Perf(O) \rightarrow D^b(A-\mod)$ and
$$
\Delta_O(A) := \left(
\frac{D^b(A-\mod)}{\cP(O)}
\right).
$$
The following result is well-known to specialists.

\begin{lemma}\label{L:keyMCM}
The triangulated category $D^b(O-\mod)/\Perf(O)$ is idempotent complete.
\end{lemma}

\begin{proof} By a result of Buchweitz \cite{Buchweitz}, we have an equivalence of triangulated
categories
$$
\frac{D^b(O-\mod)}{\Perf(O)}
\stackrel{\sim}\lar
\underline{\MCM}(O),
$$
where $\underline{\MCM}(O)$ is the \emph{stable category} of maximal Cohen--Macaulay modules
over $O$. Hence, it suffices to show that $\underline{\MCM}(O)$ is idempotent complete.

Since the ring $O$ is \emph{complete}, the endomorphism algebra of an  indecomposable Noetherian $O$--module is \emph{local}, see \cite[Proposition 6.10]{CurtisReiner}.
Let  $M$ be any object of  $\underline{\MCM}(O)$. Then it  admits a decomposition
 $M \cong M_1 \oplus \dots \oplus M_p$, such that  the ring   $\underline{\End}_O(M_i)$ is
 local for any $ 1 \le i \le p$ (in other words, $\underline{\MCM}(O)$ is a \emph{local} category).
 Hence,
$\underline{\MCM}(O)$ is idempotent complete, see for example \cite[Corollary 13.9]{KrullSchmidt}.
\end{proof}

\noindent
The main result of this section is the following.

\begin{theorem}\label{T:ResultsonVerdierQout}
The category $\Delta_O(A)$ is idempotent complete. Moreover, if
$\gldim(A) < \infty$ then $K_0\bigl(\Delta_O(A)\bigr) \cong \ZZ^r$.
\end{theorem}

\begin{proof}
First note that we have the following long exact sequences of abelian groups
$$
K_0\bigl(\Perf(O)\bigr) \stackrel{\mathsf{can}}\lar K_0\bigl(D^b(O-\mod)\bigr) \lar
K_0\bigl(\underline{\MCM}(O)\bigr) \lar 0
$$
and
$
K_0\bigl(\Perf(O)\bigr) \stackrel{\mathsf{can}}\lar  K_0\bigl(D^b(O-\mod)\bigr) \rightarrow
K_0\bigl((\underline{\MCM}(O))^\omega\bigr) \rightarrow K_{-1}\bigl(\Perf(O)\bigr) \rightarrow \qquad$  $
K_{-1}\bigl(D^b(O-\mod)\bigr),
$
where $K_{-1}\bigl(\Perf(O)\bigr)$ and $K_{-1}\bigl(D^b(O-\mod)\bigr)$ denote
the negative K--groups of  stable Frobenius categories of Schlichting \cite[Section 4]{Schlichting},
associated with the \emph{Frobenius pairs}  $\Bigl(\Com^b\bigl(\add(O)\bigr), \Com^b_{\mathsf{ac}}\bigl(\add(O)\bigr)\Bigr)$ and $\Bigl(\Com^{-,\, b}\bigl(\add(O)\bigr), \Com^{-}_{\mathsf{ac}}\bigl(\add(O)\bigr)\Bigr)$ respectively, see
\cite[Theorem 1]{Schlichting}. By \cite[Theorem 7]{Schlichting} we have  $K_{-1}\bigl(D^b(O-\mod)\bigr) = 0$. Since by Lemma \ref{L:keyMCM} the stable category $\underline{\MCM}(O)$ is idempotent complete, we also obtain  the vanishing $K_{-1}\bigl(\Perf(O)\bigr) = 0$.
In a similar way, we have long exact sequences
$$
K_0\bigl(\cP(O)\bigr) \stackrel{\mathsf{can}}\lar K_0\bigl(D^b(A-\mod)\bigr) \lar
K_0\bigl(\Delta_O(A)\bigr) \lar 0
$$
and
$
K_0\bigl(\cP(O)\bigr) \stackrel{\mathsf{can}}\lar  K_0\bigl(D^b(A-\mod)\bigr) \rightarrow
K_0\bigl(\Delta_O(A)^\omega\bigr) \rightarrow K_{-1}\bigl(\cP(O)\bigr) \rightarrow 0.
$
Since the morphism of Frobenius pairs
$$
F \otimes_O \,-\, : \Bigl(\Com^b\bigl(\add(O)\bigr), \Com^b_{\mathsf{ac}}\bigl(\add(O)\bigr)\Bigr)
\lar
\Bigl(\Com^b\bigl(\add(F)\bigr), \Com^b_{\mathsf{ac}}\bigl(\add(F)\bigr)\Bigr)
$$
induces an equivalence of the corresponding stable Frobenius categories,
\cite[Proposition 7]{Schlichting} implies that $K_{-1}\bigl(\cP(O)\bigr) = K_{-1}\bigl(\Perf(O)\bigr) = 0$. Hence,
the canonical homomorphism of abelian groups
$K_0\bigl(\Delta_O(A)\bigr)  \rightarrow K_0\bigl(\Delta_O(A)^\omega\bigr)$
is an isomorphism. By a result of Thomason \cite[Theorem 2.1]{Thomason},
the canonical functor $\Delta_O(A) \rightarrow \Delta_O(A)^\omega$ is an equivalence
of triangulated categories, i.e.~the triangulated category $\Delta_O(A)$ is idempotent complete.

Since $O$ is a complete ring,   $A$ is semi-perfect with $r+1$ pairwise non-isomorphic
 indecomposable  projective
modules.
If $\gldim(A) < \infty$ then \cite[Proposition 16.7]{CurtisReiner} implies that
$K_0\bigl(D^b(A-\mod)\bigr) \cong K_0\bigl(A-\mod\bigr) \cong  \ZZ^{r+1}$. Moreover, the image of the canonical homomorphism
$\mathsf{can}: K_0\bigl(\cP(O)\bigr) \rightarrow K_0\bigl(D^b(A-\mod)\bigr)$
is the free abelian group generated by the class of the projective module $F$. Hence,
$K_0\bigl(\Delta_O(A)\bigr) \cong \coker(\mathsf{can}) \cong  \ZZ^{r}$.
\end{proof}

\section{Description of the category $\Delta_{\mathsf{nd}}$}\label{S:ReprTheory}
\noindent
In this section $A_{\mathsf{nd}}$ denotes the arrow ideal completion of the path algebra of the following
quiver with relations $\vec{Q}_{\mathsf{nd}}$
\begin{equation}\label{E:nodalquiver}
\begin{xy}\SelectTips{cm}{}
\xymatrix
{- \ar@/^/[rr]^{ \ro }  & &  \ast \ar@/^/[ll]^{ \lu }
 \ar@/_/[rr]_{ \ru }
 & &
\ar@/_/[ll]_{ \lo } +}\end{xy}  \qquad  \ru   \ro  = 0, \quad   \lu   \lo  = 0.
\end{equation}
\begin{remark}
Note that $A_{\mathsf{nd}} = \End_{O_{\mathsf{nd}}}\bigl(O_{\mathsf{nd}} \oplus \kk\llbracket u\rrbracket
\oplus \kk\llbracket v\rrbracket\bigr)$ is the \emph{Auslander algebra} of the nodal curve singularity $O_{\mathsf{nd}}= \kk\llbracket u, v\rrbracket/uv$. In particular $\gldim(A_{\mathsf{nd}})=2$, see  \cite{AuslanderRoggenkamp} or \cite[Remark 1]{Tilting}.
\end{remark}
\noindent
Our goal is to study  the triangulated Verdier quotient category
$$
\Delta_{\mathsf{nd}}:=\frac{D^b(A_{\mathsf{nd}}-\mod)}{\Hot^b\bigl(\add(P_{*})\bigr)} \cong \frac{\Hot^b\bigl(\Pro(A_{\mathsf{nd}})\bigr)}{\Hot^b\bigl(\add(P_{*})\bigr) } \cong
\Delta_{O_{\mathsf{nd}}}(A_{\mathsf{nd}}),
$$
where $P_{*}$ is the indecomposable projective $A_{\mathsf{nd}}$--module corresponding to the vertex $*$.
By Theorem \ref{T:ResultsonVerdierQout} we know that
$\Delta_{\mathsf{nd}}$ is idempotent complete and $K_{0}(\Delta_{\mathsf{nd}})\cong \bigl\langle [P_{-}], [P_{+}]\bigr\rangle \cong \ZZ^2$.
\begin{definition} \label{D:MinS}
Let $\sigma, \tau \in \{\li, \re\}$ and $l \in \mathbb{N}$. A \emph{minimal string} $\mathcal{S}_{\tau}(l)$ is a complex of indecomposable projective $A_{\mathsf{nd}}$--modules
$$
\begin{xy}\SelectTips{cm}{}
\xymatrix{
\cdots\ar[r] & 0 \ar[r] & P_{\sigma} \ar[r] & P_{\ast} \ar[r] & \cdots \ar[r] &P_{\ast} \ar[r] &P_{\tau} \ar[r]& 0 \ar[r] &\cdots}
\end{xy}
$$
of length $l+2$ with differentials given by non-trivial paths of minimal possible
length and $P_{\tau}$ located in degree $0$. Note, that $\sigma$ is uniquely determined by $\tau$ and $l$:
$$\begin{cases} \sigma=\tau \quad \text{  if } l \text{ is even,}   \\  \sigma \ne \tau \quad \text{  if } l  \text{ is odd.} \end{cases}$$
\end{definition}

\begin{example} The two complexes depicted below are minimal strings:
\begin{itemize}
\item $\mathcal{S}_{\re}(1)=\begin{xy}\SelectTips{cm}{}\xymatrix{
\cdots \ar[r]& 0  \ar[r] & \, 0 \, \,  \ar[r] & P_{\li} \ar[r]^{\cdot  \lu } & P_{\ast} \ar[r]^{\cdot  \lo }  &P_{\re} \ar[r]& 0 \ar[r] &\cdots}\end{xy}$
\item $\mathcal{S}_{\re}(2)=\begin{xy}\SelectTips{cm}{}\xymatrix{
\cdots\ar[r] & 0 \ar[r] & P_{\re} \ar[r]^{\cdot  \ru } & P_{\ast} \ar[r]^{\cdot  \ro  \lu } &P_{\ast} \ar[r]^{\cdot  \lo }  &P_{\re} \ar[r]& 0 \ar[r] &\cdots}\end{xy}$
\end{itemize}
\end{example}

\noindent
It is interesting to note that the images of minimal strings remain to be  \emph{indecomposable} in $\Delta_{\mathsf{nd}}$. In order to prove this, we need the following result of Verdier {\cite[Proposition II.2.3.3]{Verdier}}, playing a key role in the sequel.

\begin{lemma} \label{L:Verdier}
Let $\kT$ be a triangulated category and let $\kU \subseteq \kT$ be a full triangulated subcategory. Let $Y$ be an object in $^{\perp}\kU = \bigl\{T \in \Ob(\kT) \big| \Hom_{\kT}(T, \kU)=0\bigr\}$ and let
$$
\PP\colon  \Hom_{\kT}(Y, X) \longrightarrow \Hom_{\kT/\kU}(Y, X)
$$
be the map induced by the localization functor.
Then $\PP$ is bijective for all $X$ in $\kT$.
\end{lemma}

\noindent
Of course, there is a dual result for $Y$ in $\kU^{\perp}$.

\begin{lemma}\label{L:MinS}
Let $\tau \in \{+, -\}$ and $l \in \mathbb{N}$. Then any minimal string
$\kS=\kS_{\tau}(l)$ belongs to  $\Ob\bigl(^\perp\hspace{-1pt}\Hot^b\bigl(\add(P_{*})\bigr) \bigr) \cap \Ob\bigl(\Hot^b\bigl(\add(P_{*})\bigr)^\perp\bigr)$. Moreover,   $\kS$ is indecomposable in $\Delta_{\mathsf{nd}}.$
\end{lemma}
\begin{proof} First note that $\End_{\Hot^b(\Pro (A_{\mathsf{nd}}))}(\kS) \cong \kk$. In particular,
$\kS$ is indecomposable in $\Hot^b\bigl(\Pro (A_{\mathsf{nd}})\bigr)$. Next, it is easy to check that for all $m \in \ZZ$
\[\Hom_{\Hot^b(\Pro(A_{\mathsf{nd}}))}\bigl(P_{*}[m], \kS\bigr)=0=\Hom_{\Hot^b(\Pro(A_{\mathsf{nd}}))}\bigl(\kS, P_{*}[m]\bigr) \] holds.
Now, Verdier's Lemma \ref{L:Verdier} implies indecomposability of $\kS$ in $\Delta_{\mathsf{nd}}$.
\end{proof}

\begin{definition}
Let $\kT$ be an idempotent complete  triangulated category and $X_{1}, \cdots, X_{n} \in \mathsf{Ob}(\kT)$ an arbitrary collection
of objects. Then
 $\mathsf{Tria}(X_{1}, \cdots, X_{n}) \subseteq \kT$ is the smallest full triangulated subcategory
 of $\kT$ containing all $X_i$ and closed under taking direct summands.
\end{definition}

\begin{remark}\label{R:Tria}
The projective resolutions of the simple $A_{\mathsf{nd}}$--modules $S_{+}$ and $S_{-}$ are
$$
0 \rightarrow  P_{-} \stackrel{\cdot \beta}\longrightarrow  P_{*} \stackrel{\cdot \gamma}\longrightarrow  P_{+} \rightarrow S_{+} \rightarrow  0 \quad \text{and} \quad  0 \rightarrow  P_{+} \stackrel{\cdot \delta}\longrightarrow  P_{*} \stackrel{\cdot \alpha}\longrightarrow  P_{-} \rightarrow  S_{-} \rightarrow 0.
$$
Thus $S_{\pm} \cong \mathcal{S}_{\pm}(1)$ are minimal strings. Let $\rho, \sigma, \tau \in \{\li, \re\}$ and $l \in \mathbb{N}$. The cone of
$$
\begin{xy}\SelectTips{cm}{}
\xymatrix{
\kS_{\tau}(l)  &&&0 \ar[r]& P_{\sigma} \ar[r]^{\cdot d_{3}} \ar[d]^{\mathsf{id}}& P_{*} \ar[r]^{\cdot d_{4}} &\cdots \ar[r]^{\cdot d_{l+2}}& P_{*} \ar[r]^{\cdot d_{l+3}} & P_{\tau} \ar[r] & 0 \\
\kS_{\sigma}(1)[l+1] & 0 \ar[r] & P_{\rho} \ar[r]^{\cdot d_{1}} & P_{*} \ar[r]^{\cdot d_{2}} & P_{\sigma} \ar[r] & 0
} \end{xy}
$$ is isomorphic to the following minimal string
$$
\begin{xy}\SelectTips{cm}{}
\xymatrix{
\kS_{\tau}(l+1)[1] &0 \ar[r]& P_{\rho} \ar[r]^{\cdot d_{1}} & P_{*} \ar[r]^{\cdot d_{2}d_{3}} &\cdots \ar[r]^{\cdot d_{l+2}}& P_{*} \ar[r]^{\cdot d_{l+3}} & P_{\tau} \ar[r] & 0. &&&
}\end{xy}
$$
Hence, the minimal strings are generated by $S_{+}$ and $S_{-}$. In other words, $\kS_{\tau}(l)[n]$ is contained in  $\mathsf{Tria}(S_{+}, S_{-}) \subseteq D^b(A_{\mathsf{nd}}-\mod),$ for all $\tau \in \{\li, \re\}$, $l \in \mathbb{N}$ and $n \in \ZZ$.
\end{remark}

\begin{theorem}\label{T:MainT} We use the notations from above.

\noindent
(a) Let $X$ be an indecomposable complex in $\Hot^b\bigl(\Pro(A_{\mathsf{nd}})\bigr).$ Then the  image of $X$ in $\Delta_{\mathsf{nd}}$ is either zero or isomorphic to one of the following objects
 $$ P_{\sigma}[n] \oplus P_{\tau}[m], \quad P_{\tau}[n] \quad {or} \quad \kS_{\tau}(l)[n], \quad
\text{ where } m, n \in \mathbb{Z},  l \in \mathbb{N} \text{ and } \sigma, \tau \in \{+, -\}.
 $$

\noindent
(b) Let $\sigma, \tau \in \{+, -\}$ and $n \in \ZZ$. We have the following formula:
$$\Hom_{\Delta_{\mathsf{nd}}}\bigl(P_{\sigma}, P_{\tau}[n]\bigr) \cong \begin{cases} \kk \quad \text{  if }n \in 2\ZZ_{\leq 0} \text{ and } \sigma=\tau, \\ \kk \quad \text{  if }n \in 2\ZZ_{\leq 0}-1 \text{ and }\sigma \ne \tau, \\ 0 \quad \text{  otherwise.}\end{cases}$$
In particular, $P_{\sigma}[n]$ is indecomposable in $\Delta_{\mathsf{nd}}$ for any $\sigma \in
\{+, -\}$ and $n \in \ZZ$.

\noindent
(c) Two objects from the set $\bigl\{P_{\sigma}[n], \kS_{\tau}(l)[m] \, \big| \, \sigma, \tau \in \{+, -\}, \, n, m \in \ZZ, \, l \in \mathbb{N}\bigr\}$ are isomorphic  in $\Delta_{\mathsf{nd}}$ if and only if their discrete parameters coincide.
\end{theorem}
\begin{proof} Since $A_{\mathsf{nd}}$ is a \emph{nodal} algebra, by the  work  of Burban and Drozd
 \cite{Nodal} the indecomposable objects in $\Hot^b\bigl(\Pro(A_{\mathsf{nd}})\bigr)$ are explicitly known.
They are
\begin{itemize}
\item Band objects. These are contained in $\Hot^b\bigl(\add(P_{*})\bigr)$ and thus are zero  in $\Delta_{\mathsf{nd}}$.
\item String objects.
\end{itemize}
The string objects in $\Hot^b\bigl(\Pro(A_{\mathsf{nd}})\bigr)$ can be described in the following way. Let $\ZZ\vec{A}_{\infty}^{\infty}$ be the oriented graph obtained by orienting the edges in a $\ZZ^2$--grid as indicated in Example \ref{E:Grid} below.
Let $\vec{\theta} \subseteq \ZZ \vec{A}_{\infty}^{\infty}$ be a finite oriented subgraph of type $A_{n}$ for a certain $n \in \mathbb{N}$.
 Let $\Sigma$ and $T$ be the terminal vertices of $\vec\theta$ and $\sigma, \tau \in \{\li, *, \re\}$. We insert the projective
 modules  $P_{\sigma}$ and  $P_{\tau}$ at the vertices $\Sigma$ and $T$ respectively. Next, we plug
 in $P_{\ast}$ at all intermediate vertices of $\vec\theta$.
  Finally, we put maps (given by multiplication with non-trivial paths in $\vec{Q}_{\mathsf{nd}}$) on the arrows between the corresponding indecomposable projective modules. This has to be done
  in such a way that the composition of two subsequent arrows is always zero. Additionally, at the vertices
  where $\vec\theta$ changes orientation, the inserted paths  have to  be ``alternating'', i.e.~if one adjacent path
  involves $\alpha$ or $\beta$ then the second should involve $\gamma$ or $\delta$.
    Taking a direct sum of modules and maps in every  column of the constructed diagram,
     we get a complex of projective $A_{\mathsf{nd}}$--modules $\kS$, which we shall simply call \emph{string}.
\begin{example}\label{E:Grid}
\[
\begin{xy}\SelectTips{cm}{}
\xymatrix@!=1pc
{
& & \circ \ar[rd] & & P_{\li} \ar[rd]^{\cdot \lu(\ro\lu)^n} && \circ \\
\ar[ddd] &  \circ \ar[ru] \ar[rd]&&  \circ \ar[ru] \ar[rd] && P_{*} \ar[ru] \ar[rd]     \\
&&   \circ \ar[ru] \ar[rd] && P_{*} \ar[ru]^{\cdot(\lo\ru)^m} \ar[rd] && \circ  \\
&   \circ \ar[ru] \ar[rd]&&  P_{*} \ar[ru]^{\cdot(\ro\lu)^l} \ar[rd]^{\cdot(\lo\ru)^k} && \circ \ar[ru] \ar[rd]   \\
&& \circ \ar[ru]  & & P_{*} \ar[ru] &&\circ  \\
\kS = &\cdots \ar[r] &0\ar[r]&P_{*} \ar[r]^{d_{1}} & { \underset{\displaystyle P_{*}^{\oplus 2}}{\overset{\displaystyle P_{-}}{\oplus}}}\ar[r]^{d_{2}} &
P_{*} \ar[r] & 0 \ar[r] &\cdots }\end{xy}
\]
where $d_{1}=\begin{pmatrix} 0 & \cdot(\ro\lu)^l & \cdot(\lo\ru)^k\end{pmatrix}^{\mathsf{tr}}$ and $d_{2}=\begin{pmatrix} \cdot\lu(\ro\lu)^n & \cdot(\lo\ru)^m & 0 \end{pmatrix}.$
\end{example}
\noindent
Note, that the strings with $P_{\sigma}=P_{\tau}=P_{*}$ vanish in $\Delta_{\mathsf{nd}}$. Therefore,
in what follows we may and shall  assume that $\sigma$ or $\tau \in \{\li, \re\}$.

\medskip
\noindent
\textit{Proof of (a).}
Let $\kS \in {\rm Ob}\bigl(\Hot^b\bigl(\Pro(A_{\mathsf{nd}})\bigr)\bigr)$ be an indecomposable string as defined above.

\noindent
1. If $P_{\tau}=P_{*}$  and $ \vec\theta=\Sigma \rightarrow \cdots$ hold, then there exists a distinguished triangle
$$
\kS \stackrel{f}\lar   P_{\sigma}[n] \lar \mathsf{cone}(f) \lar \kS[1]
$$
with $\mathsf{cone}(f) \in {\rm Ob}\bigl(\Hot^b\bigl(\add(P_{*})\bigr)\bigr)$, yielding an isomorphism $\kS \cong P_{\sigma}[n]$ in $\Delta_{\mathsf{nd}}$. Similarly, if $\vec\theta = \Sigma \leftarrow \cdots$ holds, then we obtain a triangle
$
P_{\sigma}[n] \stackrel{f}\lar    \kS \lar \mathsf{cone}(f) \lar  P_{\sigma}[n+1]
$
with $\mathsf{cone}(f) \in \Ob\bigl(\Hot^b\bigl(\add(P_{*})\bigr)\bigr)$ and hence an isomorphism $\kS \cong P_{\sigma}[n]$ in $\Delta_{\mathsf{nd}}.$

\medskip
\noindent
2. We may assume that $\sigma, \tau \in \{\li, \re\}$. If the graph $\vec{\theta}$ defining
$\kS$ is \emph{not} linearly oriented (i.e.~contains a subgraph ($\star$) $\begin{xy}\SelectTips{cm}{}\xymatrix{\circ \ar[r] & \circ & \circ \ar[l]}\end{xy}$ or ($\star\star$) $\begin{xy}\SelectTips{cm}{}\xymatrix{\circ & \circ \ar[r] \ar[l] &\circ}\end{xy}$),
then there exists a distinguished triangle  of the following form
$$
\begin{xy}\SelectTips{cm}{}
\xymatrix{
(\star) & P_{*}[s] \ar[r] &\kS \ar[r]& \kS' \oplus \kS'' \ar[r]& P_{*}[s+1]&  \\
(\star\star) & P_{*}[s-1] \ar[r] &\kS' \oplus \kS'' \ar[r] & \kS \ar[r] & P_{*}[s] }\end{xy}$$
and therefore $\kS \cong \kS' \oplus \kS'' \cong P_{\sigma}[n] \oplus P_{\tau}[m]$ is decomposable in $\Delta_{\mathsf{nd}}$.

\medskip
\noindent
3. Hence, without loss of generality, we may assume
$\sigma, \tau \in \{\li, \re\}$ and $\vec{\theta}$ to be  linearly oriented.
If $\kS$ has a ``non-minimal'' differential $d=\cdot p$ (i.e. the path $p$ in $\vec{Q}_{\mathsf{nd}}$ contains $\ru \lo$ or $\lu \ro$ as a subpath), then we consider the following morphism of complexes
$$\begin{xy}\SelectTips{cm}{}\xymatrix{\kS' \ar[d]_f & P_{\sigma} \ar[r] &P_{*} \ar[r] & \cdots \ar[r] & P_{*}  \ar[d]^d\\
\kS'' &&&& P_{*} \ar[r] & P_{*} \ar[r] & \cdots \ar[r] &P_{*} \ar[r] & P_{\tau}}\end{xy}$$
which can be  completed to a distinguished triangle in $\Hot^b\bigl(\Pro(A_{\mathsf{nd}})\bigr)$
$$
\kS' \stackrel{f}\lar  \kS'' \lar \kS \lar \kS'[1].
$$ By our assumption on $d$,  the morphism $f$ factors through $P_{*}[s]$ for some $s \in \ZZ$ and therefore vanishes in $\Delta_{\mathsf{nd}}$. Hence, we have  a decomposition $\kS \cong \kS'[1] \oplus \kS'' \cong P_{\sigma}[n] \oplus P_{\tau}[m]$ in $\Delta_{\mathsf{nd}}$.

\medskip
\noindent
4. If $\sigma, \tau \in \{\li, \re\}$, $\vec{\theta}$ is linearly oriented and $\kS$ has only minimal differentials, then $\kS$ is a minimal string. This  concludes the proof of part (a) of Theorem \ref{T:MainT}.

\noindent
\textit{Proof of (b).} Every morphism $P_{\sigma} \rightarrow P_{\tau}[n]$ in $\Delta_{\mathsf{nd}}$ is given by a roof $P_{\sigma} \xleftarrow{f} Q \xrightarrow{g} P_{\tau}[n]$, where $f, g$ are morphisms in $\Hot^b\bigl(\Pro(A_{\mathsf{nd}})\bigr)$ and $\mathsf{cone}(f) \in {\rm Ob}\bigl(\Hot^b\bigl(\add (P_{*})\bigr)\bigr)$. By a common abuse of terminology, we call $f$ a \emph{quasi-isomorphism}. Our aim is to find a convenient representative in each equivalence class of roofs. It turns out that $\sigma \in \{+, -\}$ and $n \in \ZZ$ determine $Q$ and $f$ of our representative and $g$ is either $0$ or determined by $\tau$ up to  scalar.

\medskip
\noindent
1. Without loss of generality, we may assume that $Q$ has no direct summands from
 $\Hot^b\bigl(\add (P_{*})\bigr)$. Indeed, if
 $Q \cong Q' \oplus Q''$ with $Q'' \in \Ob\bigl(\Hot^b\bigl(\add (P_{*})\bigr)\bigr)$, then the diagram
$$
\begin{xy}\SelectTips{cm}{}
\xymatrix{&&Q' \ar[rrd]^{g'} \ar[lld]_{f'} \ar[d]^{\left(\bsm \mathsf{id} \\ 0 \esm\right)} \\
P_{\sigma} && Q' \oplus Q'' \ar[ll]_{f=\left(\bsm f' & f'' \esm\right)} \ar[rr]^{g=\left(\bsm g' & g'' \esm\right)} && P_{\tau}[n]}
\end{xy}
$$
yields  an equivalence of roofs $P_{\sigma} \xleftarrow{f} Q \xrightarrow{g} P_{\tau}[n]$ and
$P_{\sigma} \xleftarrow{f'} Q' \xrightarrow{g'} P_{\tau}[n]$.

\medskip
\noindent
2. Using our assumptions on $f$ and $Q$ in conjunction with the description of indecomposable strings in $\Hot^b\bigl(\Pro(A_{\mathsf{nd}})\bigr)$, it is not difficult to see that  $Q$  can (without restriction) taken to be an indecomposable string with $\tau=*$ and $f$ to be of the following form:
$$
\begin{xy}\SelectTips{cm}{}
\xymatrix@R=0.3pc{
&\cdots \ar[r] & P_{*}  \ar[r]& P_{*} \ar[r] & P_{*} \ar[rdd] \\ Q \ar[ddd]_f \\
&P_{\sigma} \ar[dd]_{\mathsf{id}} \ar[r] & P_{*} \ar[r] & \cdots \ar[r] & P_{*} \ar[r] & P_{*} \\ \\
P_{\sigma }&P_{\sigma}}
\end{xy}
$$

\noindent
3.  Without loss of generality, we may assume that $Q$ is constructed from a linearly oriented graph $\vec{\theta}$. Indeed, otherwise we may consider the truncated complex $Q^{\leq}$ defined in the diagram below and replace our roof by an equivalent one.
$$
\begin{xy}\SelectTips{cm}{}
\xymatrix@R=0.3pc{
&\cdots \ar[r]& P_{*} \ar[r]& P_{*} \ar[r] & P_{*} \ar[rdd] \\
Q \\
&P_{\sigma} \ar[r]^{d_{1}} & P_{*} \ar[r]^{d_{2}} & \cdots \ar[r]^{d_{n-1}} & P_{*} \ar[r]^{d_{n}} & P_{*} \\ \\
Q^{\leq} \ar[uuu]^q&P_{\sigma} \ar[uu]^{\mathsf{id}}\ar[r]^{d_{1}} & P_{*} \ar[r]^{d_{2}} \ar[uu]^{\mathsf{id}} & \cdots \ar[r]^{d_{n-1}} & P_{*} \ar[r]^{d_{n}} \ar[uu]^{\mathsf{id}} & P_{*} \ar[uu]^{\mathsf{id}} }
\end{xy}
$$
$$
\begin{xy}\SelectTips{cm}{}
\xymatrix{&&Q^\leq \ar[rrd]^{gq} \ar[lld]_{fq} \ar[d]^{q} \\
P_{\sigma} &&Q \ar[ll]_{f} \ar[rr]^{g} && P_{\tau}[n]}
\end{xy}
$$
In particular,  $n >0 $ implies that $\Hom_{\Delta_{\mathsf{nd}}}\bigl(P_{\sigma}, P_{\tau}[n]\bigr)=0$ holds.

\medskip
\noindent
4. By the above reductions, $g$ has the following form:
$$
\begin{xy}\SelectTips{cm}{}
\xymatrix{
Q \ar[d]^g &  P_{\sigma} \ar[r] & P_{*} \ar[r] & \cdots \ar[r] & P_{*} \ar[r] & P_{*} \ar[r]\ar[d]^{g} & P_{*} \ar[r] & \cdots \\
P_{\tau}[n] & &&&& P_{\tau}
}\end{xy}
$$
We may truncate again so that $Q$ ends at degree $-n$:
$$
\begin{xy}\SelectTips{cm}{}
\xymatrix{&&Q^{\leq -n} \ar[rrd] \ar[lld]  \\
P_{\sigma} &&Q \ar[ll] \ar[rr] \ar[u]&& P_{\tau}[n]}
\end{xy}
$$

\noindent
5. Next, we may  assume that $Q$ has minimal differentials (see  Definition \ref{D:MinS}) and thus is  uniquely determined by $\sigma$ and $n$. Indeed, otherwise there exists  a quasi-isomorphism:
$$
\begin{xy}\SelectTips{cm}{}
\xymatrix{
Q'  \ar[d]^q & P_{\sigma} \ar[r] \ar[d]^{\mathsf{id}}& P_{*} \ar[r] \ar[d]^{\mathsf{id}} & \cdots \ar[r] &P_{*}  \ar[r]^{d'} \ar[d]^{\mathsf{id}}& P_{*}  \ar[d]^{d''}\ar[r] &0 \ar[d]  \\
Q & P_{\sigma} \ar[r] & P_{*} \ar[r] & \cdots \ar[r] & P_{*} \ar[r]^d & P_* \ar[r] & P_* \ar[r] & \cdots
}\end{xy}
$$

\noindent
6. Summing up, our initial roof can be replaced by an equivalent one  of  the following form
$$
\begin{xy}\SelectTips{cm}{}
\xymatrix{
P_{\sigma} & P_{\sigma} \\
Q \ar[u]^f \ar[d]_g & P_{\sigma} \ar[u]^{\mathsf{id}} \ar[r] & P_{*} \ar[r] & \cdots \ar[r] &P_{*} \ar[d]_g \\
P_{\tau}[n] & &&& P_{\tau}
}
\end{xy}
$$
If $g$ is not minimal (i.e.~not given by multiplication with a single arrow), then it factors over $P_{*}[n]$ and therefore vanishes in $\Delta_{\mathsf{nd}}$. Thus the morphism space $\Hom_{\Delta_{\mathsf{nd}}}(P_{\sigma}, P_{\tau}[n])$ is at most one dimensional. Moreover, $g$ can be non-zero only if $n$ has the right parity.

\medskip
\noindent
7. Consider a roof $P_{\sigma} \xleftarrow{f} Q \xrightarrow{g} P_{\tau}[n]$
as in the previous step and assume that $g$ is non-zero and minimal. We want to show that the roof defines a non-zero homomorphism in $\Delta_{\mathsf{nd}}$. We have a triangle
$ Q \stackrel{g}\rightarrow  P_{\tau}[n] \rightarrow  \String{\tau}{-n}{n} \rightarrow Q[1]$ in $\Hot^b\bigl(A_{\mathsf{nd}}-\mod\bigr)$ yielding a triangle
$P_{\sigma} \rightarrow  P_{\tau}[n] \rightarrow  \String{\tau}{-n}{n} \rightarrow P_{\sigma}[1]$ in $\Delta_{\mathsf{nd}}$. Since $\String{\tau}{-n}{n}$ is indecomposable,  the map is non-zero. The claim follows.

\medskip
\noindent
\textit{Proof of (c).}
 Note that for  $X \in \Ob\bigl(\mathsf{Tria}(S_{+}, S_{-})\bigr)$ we have $[X] =   n \cdot  \bigl([P_{+}]+[P_{-}]\bigr) \in K_{0}(\Delta_{\mathsf{nd}})$ for a certain $n \in \mathbb{Z}$. Thus, the images of indecomposable projective $A_{\mathsf{nd}}$--modules $P_{+}$ and $P_{-}$ are not contained in $\mathsf{Tria}(S_{+}, S_{-})$.
 By Lemma \ref{L:Verdier} and the classification of indecomposable strings in $\Hot^b\bigl(\Pro(A_{\mathsf{nd}})\bigr)$, it remains to show that $P_{\sigma}[n] \cong P_{\tau}[m]$ implies $\sigma=\tau$ and $n=m$. Assume that $n>m$ holds. Then using Lemma \ref{L:Verdier} again, we obtain
$$\Hom_{\Delta_{\mathsf{nd}}}\bigl(P_{\sigma}[n], \String{\sigma}{1}{n}\bigr)\cong \kk \ne 0 =\Hom_{\Delta_{\mathsf{nd}}}\bigl(P_{\tau}[m], \String{\sigma}{1}{n}\bigr).$$ This is a contradiction. Similarly, the assumption $\sigma \ne \tau$ leads to a contradiction.
\end{proof}

\begin{remark}
Theorem \ref{T:MainT} and Lemma \ref{L:Verdier} reduce the computation of morphism spaces in $\Delta_{\mathsf{nd}}$ to a computation in $\Hot^b\bigl(\Pro(A_{\mathsf{nd}})\bigr)$.
Moreover, every minimal string may be presented as a cone of a morphism $P_{\sigma}[n] \rightarrow P_{\tau}[m]$ in $\Delta_{\mathsf{nd}}$ (see step 7 in the proof of Theorem \ref{T:MainT} (b)). Using this fact and the long exact $\Hom$-sequence, one can show that
$\dim_{\kk}\Hom_{\Delta_{\mathsf{nd}}}(X, Y) \leq 1$ holds for all indecomposable objects $X$ and  $Y$  in  $\Delta_{\mathsf{nd}}$.
\end{remark}

\begin{corollary}\label{C:indinTria}
The indecomposable objects of the triangulated subcategory $\mathsf{Tria}(S_-, S_+) \subset
D^b\bigl(A_{\mathsf{nd}}-\mod\bigr)$ are precisely the shifts of the minimal strings $\kS_\tau(l)$.
\end{corollary}

\begin{proof}
By Remark \ref{R:Tria}, we know that all minimal strings belong to $\mathsf{Tria}(S_{-}, S_+)$.
Hence, we just have to prove that there are no other indecomposable objects.
According to Lemma \ref{L:Verdier} and  Lemma \ref{L:MinS},
the functor $\mathsf{Tria}(S_{-}, S_+) \rightarrow \Delta_{\mathsf{nd}}$ is fully faithful.
Therefore, the indecomposable objects of the category $\mathsf{Tria}(S_{-}, S_+)$ and its essential
image in $\Delta_{\mathsf{nd}}$
are the same. By Theorem \ref{T:MainT}, all indecomposable objects of $\Delta_{\mathsf{nd}}$ are known and
the shifts of the objects  $P_{+}$ and $P_{-}$ are not contained in $\mathsf{Tria}(S_{-}, S_{+})$. Hence,
the minimal strings are the only indecomposable objects of $\mathsf{Tria}(S_{-}, S_+)$.
\end{proof}

\section{Connection with the category $\bigl(D^b(\Lambda-\mod)/\Band(\Lambda)\bigr)^\omega$}\label{S:Quivers}
\noindent
Let  $\Lambda$  be the path algebra of the following quiver with relations
\begin{equation}\label{E:gentlequiver}
\begin{xy}\SelectTips{cm}{}
\xymatrix
{
1 \ar@/^/[rr]^{a} \ar@/_/[rr]_{c} & & 2 \ar@/^/[rr]^{b} \ar@/_/[rr]_{d} & & 3
}
\end{xy}
\qquad ba = 0,  \quad dc =  0
\end{equation}
and  $\Band(\Lambda)$ be the full subcategory of $D^b(\Lambda-\mod)$ consisting of those objects, which are
invariant under the Auslander--Reiten translation in $D^b(\Lambda-\mod)$. By \cite[Corollary 6]{Tilting},
the subcategory $\Band(\Lambda)$ is triangulated.
Hence, we can define the triangulated category
$$\widetilde{\Delta}_{\mathsf{nd}} := \bigl(D^b(\Lambda-\mod)/\Band(\Lambda)\bigr)^\omega,$$
i.e.~the idempotent completion of the Verdier quotient $D^b(\Lambda-\mod)/\Band(\Lambda)$ (see \cite{BalmerSchlichting}).
The main goal of this section is to show that $\widetilde{\Delta}_{\mathsf{nd}}$ and  $\Delta_{\mathsf{nd}}$ are triangle equivalent.
\begin{lemma}
The indecomposable projective $\Lambda$--modules are pairwise isomorphic in $\widetilde{\Delta}_{\mathsf{nd}}.$
\end{lemma}
\begin{proof} Complete the following exact sequences of $\Lambda$--modules
\begin{align*}
0 \longrightarrow P_{2} \longrightarrow P_{1} \longrightarrow \left(\begin{xy}\SelectTips{cm}{}\xymatrix{\kk \ar@/^/[r]^1 \ar@/_/[r]_1 & \kk \ar@/^/[r] \ar@/_/[r] &0}\end{xy}\right) \longrightarrow 0 \\
0 \longrightarrow P_{3} \longrightarrow P_{2} \longrightarrow \left(\begin{xy}\SelectTips{cm}{}\xymatrix{0 \ar@/^/[r] \ar@/_/[r] & \kk \ar@/^/[r]^1 \ar@/_/[r]_{1} & \kk}\end{xy}\right) \longrightarrow 0
\end{align*}
to triangles in $D^b(\Lambda-\mod)$ and note that the modules on the right-hand side  are bands.
\end{proof}
\noindent
Let $P \in \Ob\!\left(\widetilde{\Delta}_{\mathsf{nd}}\right)$ be the common image of the indecomposable projective $\Lambda$--modules.
\begin{lemma}
The endomorphisms of $P$, which are given by the roofs
\begin{equation}\label{E:Idempotents}
e_{+}= P_{1} \xleftarrow{\cdot(a+c)} P_{2} \xrightarrow{\cdot a} P_{1} \quad  \text{     and     } \quad e_{-}=P_{1} \xleftarrow{\cdot(a+c)}
P_{2} \xrightarrow{\cdot c} P_{1}
\end{equation}
satisfy $e_{-}e_{+}=0=e_{+}e_{-}$ and $e_{-}+e_{+}=\mathsf{id}_{P}$ and thus are \emph{idempotent}.
In particular, we have a direct sum decomposition $P \cong P^+ \oplus P^-$, where $P^{+}=(P, e_{+})$ and $P^{-}=(P, e_{-})$.
\end{lemma}
\begin{proof} It is clear that $e_{-}+e_{+}=\mathsf{id}_{P}$. The equality $e_{+}e_{-}=0$ follows
from the  diagram
\begin{align*}
e_{+}e_{-}=\begin{split}\begin{xy}\SelectTips{cm}{}\xymatrix{&&P_{3} \ar[rd]^{\cdot d} \ar[ld]_{\cdot(b+d)} \\
                              &P_{2} \ar[rd]^(.44){\cdot a} \ar[ld]_{\cdot(a+c)}&& P_{2} \ar[rd]^{\cdot c} \ar[ld]_{\cdot(a+c)} \\
                              P_{1}&&P_{1}&&P_{1}}\end{xy}\end{split}=\, \, P_{1} \xleftarrow{\cdot(da+bc)} P_{3} \xrightarrow{0} P_{1}.
\end{align*}
The second equality $e_{-}e_{+}=0$ follows from  a similar calculation. Hence,  $e_{\pm}^2 = e_{\pm}$.
\end{proof}

\noindent
Next, note the following easy but useful result.
\begin{lemma}\label{L:fractCY}
Let $\kA$ be an abelian category and let $\mathbb{S}\colon D^b(\kA) \rightarrow D^b(\kA)$ be a triangle equivalence. If $X_{1}, X_{2} \in {\rm Ob}\bigl(D^b(\kA)\bigr)$ and $n_{1}, n_{2}, m_{1}, m_{2} \in \ZZ$ satisfy
$$ \mathbb{S}^{m_{1}}X_{1} \cong X_{1}[n_{1}], \quad  \quad \mathbb{S}^{m_{2}}X_{2} \cong X_{2}[n_{2}]
 \quad \text{ and  } \quad  d=m_{1}n_{2}-m_{2}n_{1} \ne 0,$$ then
$\Hom_{D^b(\kA)}(X_{1}, X_{2}) = 0 = \Hom_{D^b(\kA)}(X_{2}, X_{1}).$
\end{lemma}
\begin{proof} By the symmetry of the claim, it suffices to show that $\Hom_{D^b(\kA)}(X_{1}, X_{2})$ vanishes. Since $\mathbb{S}$ is an equivalence, we have a chain of isomorphisms
\begin{align*}
&\Hom_{D^b(\kA)}(X_{1}, X_{2}) \cong \Hom_{D^b(\kA)}\bigl(\mathbb{S}^{\pm m_{1}m_{2}}X_{1},\mathbb{S}^{\pm m_{1}m_{2}} X_{2}\bigr)\cong \\ & \Hom_{D^b(\kA)}\bigl(X_{1}[\pm m_{2}n_{1}], X_{2}[\pm m_{1}n_{2}]\bigr)
          \cong \Hom_{D^b(\kA)}\bigl(X_{1}, X_{2}[\pm d]\bigr)  \cong \Hom_{D^b(\kA)}\bigl(X_{1}, X_{2}[\pm kd]\bigr)
\end{align*}
for all $k \in \mathbb{N}$.  Hence, the claim follows from the boundedness of $X_{1}$ and $X_{2}$ together with the fact that there are no non-trivial Ext--groups $\Ext^{-n}_{\kA}(A_{1}, A_{2}) \cong \Hom_{D^b(\kA)}(A_{1}, A_{2}[-n])$, where $A_{1}, A_{2} \in {\rm Ob}(\kA)$ and $n$ is a positive integer.
\end{proof}
\noindent
A direct calculation in $D^b(\Lambda-\mod)$ yields the following result.
\begin{lemma}\label{L:Xplusminus}
Let $\mathbb{S}\colon D^b(\Lambda-\mod) \rightarrow D^b(\Lambda-\mod)$ be the Serre functor,
$$X_{+} =\begin{xy}\SelectTips{cm}{}\xymatrix{\kk \ar@/^/[r]^1 \ar@/_/[r]_0 & \kk \ar@/^/[r]^0 \ar@/_/[r]_{1} &\kk} \end{xy} \quad \text{ and } \quad
X_{-}=\begin{xy}\SelectTips{cm}{}\xymatrix{\kk \ar@/^/[r]^{0} \ar@/_/[r]_1 & \kk \ar@/^/[r]^1 \ar@/_/[r]_{0} &\kk}\end{xy}.$$
Then $\mathbb{S}(X_{\pm}) \cong X_{\mp}[2].$
In particular, $X_{\pm}$ are $\frac{4}{2}$-fractionally Calabi--Yau objects.
\end{lemma}

\begin{corollary}
The following composition of the inclusion and projection functors
$$\mathsf{Tria}(X_{+}, X_{-}) \hookrightarrow
D^b(\Lambda-\mod) \lar
 \frac{\displaystyle D^b(\Lambda-\mod)}{\displaystyle \Band(\Lambda)}$$
 is fully faithful.
\end{corollary}
\begin{proof}
Lemma \ref{L:Xplusminus} and Lemma \ref{L:fractCY} applied to the Serre functor $\mathbb{S}$ in $D^b(\Lambda-\mod)$ imply that
$X_{\pm} \in \Ob\bigl( ^{\perp}\hspace{-0.5pt} \Band(\Lambda)\bigr) \cap \Ob\bigl(\Band(\Lambda)^{\perp}\bigr)$. Hence, the claim
 follows from Lemma \ref{L:Verdier}.
\end{proof}
\begin{theorem} \label{T:Maintilde}
There exists an equivalence of triangulated categories
$$
\GG \colon \, \frac{\displaystyle D^b(A_{\mathsf{nd}}-\mod)}{\displaystyle \Hot^b\bigl(\add(P_{*})\bigr)} \longrightarrow
\left(\frac{\displaystyle D^b(\Lambda-\mod)}{\displaystyle \Band(\Lambda)}\right)^{\omega}.
$$
\end{theorem}
\begin{proof}
Let $E = V(zy^2 - x^3 - x^2 z) \subset \mathbb{P}^2$ be a nodal cubic curve and
$\kF' = \kI$ be the ideal sheaf of the singular point of $E$. Let
$\kF = \kO \oplus \kI$, $\kA = {\mathcal End}_E(\kF)$  and $\mathbb{E} = (E, \kA)$.
By a result of Burban and Drozd \cite[Section 7]{Tilting}, there exists a triangle  equivalence
$$\mathbb{T}\colon D^b\bigl(\Coh (\mathbb{E})\bigr) \lar  D^b(\Lambda-\mod)$$
identifying
the image of the category $\Perf(E)$ with the category $\Band(\Lambda)$. Moreover, by
\cite[Proposition 12]{Tilting}, the functor   $\TT$ restricts to an equivalence
$\mathsf{Tria}(S_{+}, S_{-}) \rightarrow \mathsf{Tria}(X_{+}, X_{-}).$ This can be summarized
by the following commutative diagram of categories and functors
\begin{equation}\label{E:importantdiagram}
\begin{array}{c}
\begin{xy}\SelectTips{cm}{}
\xymatrix{
\frac{\displaystyle D^b(A_{\mathsf{nd}}-\mod)}{\displaystyle \Hot^b\bigl(\add(P_{*})\bigr)} \ar@/^{10pt}/[rrd]^-\GG & & \\
\left(\frac{\displaystyle D^b(A_{\mathsf{nd}}-\fdmod)}{\displaystyle \Hot^b_{\mathsf{fd}}\bigl(\add(P_{*})\bigr)}\right)^{\omega}  \ar[r]^-{\sim} \ar[u]^-\sim &\left(\frac{\displaystyle D^b\bigl(\Coh(\mathbb{E})\bigr)}{\displaystyle \cP(E)}\right)^{\omega} \ar[r]^-\sim & \left(\frac{\displaystyle D^b(\Lambda-\mod)}{\displaystyle \Band(\Lambda)}\right)^{\omega} \\
 D^b(A_{\mathsf{nd}}-\fdmod) \ar[u]^-{\mathsf{can}} \ar[r] & D^b\bigl(\Coh(\mathbb{E})\bigr) \ar[u]^-{\mathsf{can}} \ar[r]^-{\TT}  &  D^b(\Lambda-\mod) \ar[u]_-{\mathsf{can}} \\
 & \mathsf{Tria}(S_{+}, S_{-})  \ar@{^{(}->}[u]\ar[r]^{\sim} &
 \mathsf{Tria}(X_{+}, X_{-})\ar@{^{(}->}[u]
}\end{xy}
\end{array}
\end{equation}
where $\GG\colon \Delta_{\mathsf{nd}} \rightarrow \widetilde{\Delta}_{\mathsf{nd}}$ is the induced  equivalence of triangulated categories.\end{proof}

\begin{lemma}\label{L:IndtildeDn}
The indecomposable objects of the triangulated category $\widetilde{\Delta}_{\mathsf{nd}}$ are
\begin{itemize}
\item $P^{\pm}[n] \cong \GG\bigl(P_{\pm}[n]\bigr), \quad n \in \ZZ.$
\item The indecomposables  of the full subcategory  $\mathsf{Tria}(X_{+}, X_{-})\cong \GG\bigl(\mathsf{Tria}(S_{+}, S_{-})\bigr).$
\end{itemize}
\end{lemma}
\begin{proof}
Consider the projective resolution of the simple $A_{\mathsf{nd}}$-module $S_{*}$
$$\begin{xy}\SelectTips{cm}{}\xymatrix{0 \ar[r] & P_{\li} \oplus P_{\re} \ar[r]^(0.6){\left(\bsm \cdot \beta & \cdot \delta \esm\right)} & P_{*} \ar[r] & S_{*} \ar[r] & 0}\end{xy}.$$
Completing it to a distinguished triangle yields an isomorphism
$S_{*}[-1] \cong P_{+} \oplus P_{-}$ in $\Delta_{\mathsf{nd}}$. In the notations
 of the diagrams (\ref{E:importantdiagram}) and (\ref{E:gentlequiver}), we have $\TT(S_{*}) \cong P_{3}[1]$ and therefore
$$\GG(P_{+} \oplus P_{-}) \cong \GG(S_{*}[-1]) \cong P \cong (P^+ \oplus P^-).$$
Recall that $X_{+}\cong \bigl(P_{3} \stackrel{\cdot d}\longrightarrow P_{2} \stackrel{\cdot c}\longrightarrow P_{1}\bigr)$, where $P_{1}$ is located in degree $0$ and $P^\pm := (P, e_{\pm})
\in \Ob\bigl(\widetilde{\Delta}_{\mathsf{nd}}\bigr)$, with $e_{\pm}$ as defined in (\ref{E:Idempotents}). A direct calculation shows that the obvious morphism from $P_{1}$ to $X_{+}$ induces a non-zero morphism $P^+=(P, e_{+}) \rightarrow X_{+}$ in $\widetilde{\Delta}_{\mathsf{nd}}$, whereas $\Hom_{\widetilde{\Delta}_{\mathsf{nd}}}(P^{+}, X_{-})=0$. Moreover, it was shown in \cite{Tilting} that $\TT(S_{\pm})\cong X_{\pm}$. This implies $\GG(S_{\pm})\cong X_{\pm}$ and thus $\GG(P_{\pm})\cong P^{\pm}$.
Theorem \ref{T:MainT} and Corollary \ref{C:indinTria} yield  the stated classification of indecomposables   in $\widetilde{\Delta}_{\mathsf{nd}}.$
\end{proof}

\section{Concluding remarks on $\Delta_{\mathsf{nd}}$}
\begin{proposition}\label{P:ARtriangles}
The  category $\mathsf{Tria}(S_{+}, S_{-}) \subset \Delta_{\mathsf{nd}}$ has Auslander--Reiten triangles.
\end{proposition}
\begin{proof}
 As mentioned  above, we have an exact equivalence of triangulated categories
 $$\mathsf{Tria}(S_{+}, S_{-}) \cong \mathsf{Tria}(X_{+}, X_{-}) \subset D^b(\Lambda-\mod).$$
 The category $D^b(\Lambda-\mod)$ has a Serre functor $\mathbb{S}$ and therefore has Auslander--Reiten triangles, see
  \cite{Happel}. Let $\tau=\mathbb{S}\circ[-1]$ be the Auslander--Reiten translation. Using that $\tau$ is an equivalence and Lemma \ref{L:Xplusminus}, we obtain $\tau\bigl(\mathsf{Tria}(X_{+}, X_{-})\bigr) \cong \mathsf{Tria}\bigl(\tau(X_{+}), \tau(X_{-})\bigr) \cong \mathsf{Tria}(X_{+}, X_{-}).$ Now, the restriction of $\tau$ to $\mathsf{Tria}(X_{+}, X_{-})$ is the Auslander--Reiten translation of this subcategory.
\end{proof}
\begin{remark}\label{r:AR-quiver}
One can show that the Auslander--Reiten quiver of $\mathsf{Tria}(S_{+}, S_{-})$ consists of two $\ZZ A_{\infty}$--components. We draw one of them below, indicating the action of the Auslander--Reiten translation by $\SelectTips{cm}{} \xymatrix{&\ar@{-->}[l]}$. The other component is obtained from this one by changing the roles of $+$ and $-$.
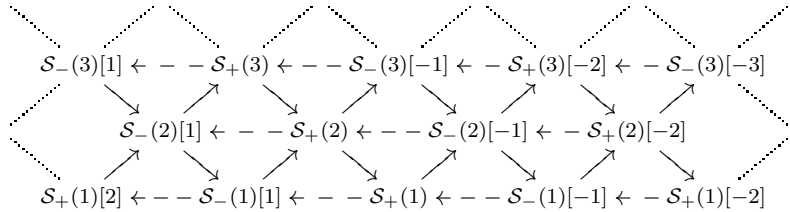
\[
{\scriptsize
\SelectTips{cm}{}
\begin{xy} 0;<0.6pt,0pt>:<0pt,-0.5pt>::
(0,100) *+{} ="0",
(0,0) *+{} ="1",
(50,150) *+{\kS_{+}(1)[2]} ="2",
(50,50) *+{\kS_{-}(3)[1]} ="3",
(100,100) *+{\kS_{-}(2)[1]} ="5",
(100,0) *+{} ="6",
(150,150) *+{\kS_{-}(1)[1]} ="7",
(150,50) *+{\kS_{+}(3)} ="8",
(200,100) *+{\kS_{+}(2)} ="10",
(200,0) *+{} ="11",
(250,150) *+{\kS_{+}(1)} ="12",
(250,50) *+{\kS_{-}(3)[-1]} ="13",
(300,100) *+{\kS_{-}(2)[-1]} ="15",
(300,0) *+{} ="16",
(350,150) *+{\kS_{-}(1)[-1]} ="17",
(350,50) *+{\kS_{+}(3)[-2]} ="18",
(400,100) *+{\kS_{+}(2)[-2]} ="20",
(400,0) *+{} ="21",
(450,150) *+{\kS_{+}(1)[-2]}="22",
(450,50) *+{\kS_{-}(3)[-3]} ="23",
(500,0) *+{} ="24",
(500,100) *+{} ="25",
"0", {\ar@{.}"2"},
"0", {\ar@{.}"3"},
"1", {\ar@{.}"3"},
"2", {\ar"5"},
"3", {\ar"5"},
"3", {\ar@{.}"6"},
"5", {\ar"7"},
"5", {\ar"8"},
"6", {\ar@{.}"8"},
"7", {\ar@{-->}"2"},
"7", {\ar"10"},
"8", {\ar@{-->}"3"},
"8", {\ar"10"},
"8", {\ar@{.}"11"},
"10", {\ar@{-->}"5"},
"10", {\ar"12"},
"10", {\ar"13"},
"11", {\ar@{.}"13"},
"12", {\ar@{-->}"7"},
"12", {\ar"15"},
"13", {\ar@{-->}"8"},
"13", {\ar"15"},
"13", {\ar@{.}"16"},
"15", {\ar@{-->}"10"},
"15", {\ar"17"},
"15", {\ar"18"},
"16", {\ar@{.}"18"},
"17", {\ar@{-->}"12"},
"17", {\ar"20"},
"18", {\ar@{-->}"13"},
"18", {\ar"20"},
"18", {\ar@{.}"21"},
"20", {\ar@{-->}"15"},
"20", {\ar"22"},
"20", {\ar"23"},
"21", {\ar@{.}"23"},
"22", {\ar@{-->}"17"},
"22", {\ar@{.}"25"},
"23", {\ar@{.}"24"},
"23", {\ar@{.}"25"},
"23", {\ar@{-->}"18"},
\end{xy}
}
\]
The category $\Delta_{\mathsf{nd}}$ does not have Auslander--Reiten triangles, but we may still consider the quiver of irreducible morphisms in $\Delta_{\mathsf{nd}}$, which has two \emph{additional} $A^{\infty}_{\infty}$--components.
$$
\begin{xy}\SelectTips{cm}{} \xymatrix{\ar@{.>}[r] & P_{\pm}[2] \ar[r] & P_{\mp}[1] \ar[r] &P_{\pm} \ar[r] & P_{\mp}[-1] \ar[r] & P_{\pm}[-2]\ar@{.>}[r] &}  \end{xy}$$

\end{remark}
\begin{proposition}
The respective triangulated categories  $\mathsf{Tria}(X_{+}, X_{-})$ and $\Delta_{\mathsf{nd}}$  are not triangle equivalent to the bounded derived category of a finite dimensional algebra.
\end{proposition}
\begin{proof}
 Assume that there exists a triangle equivalence to the derived category of a finite dimensional algebra $A$. Then $D^b(A-\mod)$ is of discrete representation type. Hence, $A$ is a \emph{gentle} algebra occuring in Vossieck's classification \cite{Vossieck}. In particular, $A$ is a Gorenstein algebra \cite{GeissReiten}. Therefore, the Nakayama functor defines a Serre functor on $\Hot^b\bigl(\Pro(A)\bigr)$ \cite{Happel}, whose action on objects is described in \cite[Theorem B]{BobinskiGeissSkowronski}. On the other hand, $\mathbb{S}^2(X) \cong X[4]$ holds for all objects $X$ in $\mathsf{Tria}(X_{+}, X_{-})$, by Lemma \ref{L:Xplusminus} and Proposition \ref{P:ARtriangles}. This yields a contradiction.\end{proof}

\noindent
The following proposition generalizes Theorem \ref{T:Maintilde}.
\begin{proposition}
Let $n\geq 1$ and $\Lambda_{n}$ be the path algebra of the following quiver
 $$
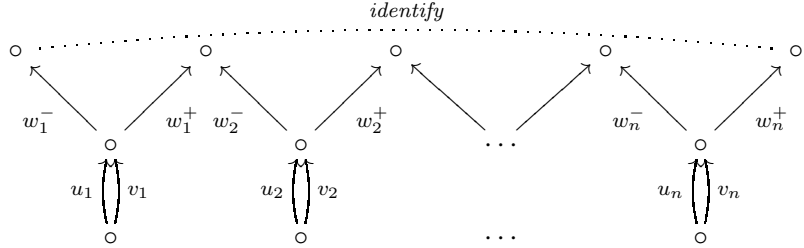
\begin{xy}\SelectTips{cm}{}\xymatrix{\circ  && \circ  && \circ  && \circ  && \circ \ar@{..}@/_8pt/[llllllll]_{\textit{identify}}
\\
& \circ  \ar[lu]^-{w_{1}^-} \ar[ru]_-{w_{1}^+} && \circ  \ar[lu]^-{w_{2}^-} \ar[ru]_-{w_{2}^+} && \cdots \ar[lu] \ar[ru] && \circ \ar[ru]_-{w_{n}^+} \ar[lu]^-{w_{n}^-}
\\
& \circ \ar@/^3pt/[u]^-{u_{1}} \ar@/_3pt/[u]_-{v_{1}} && \circ \ar@/^3pt/[u]^-{u_{2}} \ar@/_3pt/[u]_-{v_{2}} &&\cdots&& \circ
\ar@/^3pt/[u]^-{u_{n}} \ar@/_3pt/[u]_-{v_{n}}
}\end{xy}
$$
subject to the relations $w_i^- u_i = 0$ and $w_i^+ v_i=0$ for all $1 \le i \le n$. Then
$$\Delta_{n} := \left(\frac{D^b(\Lambda_{n}-\mod)}{\Band(\Lambda_{n})}\right)^{\omega} \cong \bigvee_{i=1}^n \Delta_{\mathsf{nd}}.$$
In particular, the  category $\Delta_{n}$ is representation discrete, $\Hom$-finite and $K_{0}(\Delta_{n}) \cong (\ZZ^2)^{\oplus n}$.
\end{proposition}
\begin{proof}
Let $E = E_{n}$ be a Kodaira cycle of $n$ projective lines and $\mathbb{E} = \bigl(E,
 {\mathcal End}_{E}(\kO \oplus \kI_{Z})\bigr)$, where $\kI_{Z}$ is the ideal sheaf of the singular locus $Z$. By  \cite[Proposition 10]{Tilting},  there exists an equivalence of triangulated categories
$D^b\bigl(\Coh(\mathbb{E})\bigr) \xrightarrow{\sim} D^b(\Lambda_{n}-\mod)$
identifying   $\Perf(E) \cong \cP(E) \subset D^b\bigl(\Coh(\mathbb{E})\bigr)$
with $\Band(\Lambda_{n})\subset D^b(\Lambda_{n}-\mod)$, see \cite[Corollary 6]{Tilting}. Thus,
 Corollary \ref{C:Localization} yields  the proof.
\end{proof}

\noindent
\begin{remark}
Let $A$ be the path algebra of the Kronecker quiver $\begin{xy}\SelectTips{cm}{}\xymatrix{\circ \ar@/^/[r] \ar@/_/[r] &  \circ}\end{xy}$ and
$\Band(A)$ be the full subcategory of $D^b(A-\mod)$ consisting of those objects, which are
invariant under the Auslander--Reiten translation. Note, that the objects of $\Band(A)$ are direct
sums of indecomposable objects lying in tubes. Moreover, $\Band(A)$ is closed under taking cones
and direct summands. In particular, it is a triangulated subcategory. It is interesting to note,
that the Verdier quotient category $D^b(A-\mod)/\Band(A)$ is \emph{not} $\Hom$-finite.

Indeed,  the well-known tilting equivalence $D^b(A-\mod) \rightarrow D^b\bigl(\Coh(\PP^1)\bigr)$ identifies
$\Band(A)$ with the category $D^b\bigl(\mathsf{Tor}(\PP^1)\bigr)$, where $\mathsf{Tor}(\PP^1)$
is the category of torsion coherent sheaves on $\PP^1$. Hence, by Miyachi's theorem \cite{Miyachi} we have:
$$ \frac{D^b(A-\mod)}{\Band(A)} \cong \frac{D^b\bigl(\Coh(\PP^1)\bigr)}{D^b\bigl(\mathsf{Tor}(\PP^1)\bigr)} \cong D^b\!\left(\frac{\Coh(\PP^1)}{\mathsf{Tor}(\PP^1)}\right) \cong D^b\bigl(\kk(t)-\mod\bigr),$$
 where $\kk(t)$ is the field of rational functions. Therefore, the category
 $D^b(A-\mod)/\Band(A)$ is not $\Hom$--finite.
\end{remark}

\noindent
We conclude this paper by giving a relation between our non-commutative singularity category $\Delta_{\mathsf{nd}}$ and the (classical) singularity category $\underline{\mathsf{MCM}}(O_{\mathsf{nd}})$
for the  ring $O_{\mathsf{nd}} = \kk\llbracket u, v\rrbracket/uv$.

\begin{proposition}\label{P:Quotient}
There is an equivalence of triangulated categories
$$\frac{\Delta_{\mathsf{nd}}}{\mathsf{Tria}(S_{+}, S_{-})} \stackrel{\sim}\longrightarrow \underline{\mathsf{MCM}}(O_{\mathsf{nd}}).$$
\end{proposition}
\begin{proof}
The functor $\mathsf{Hom}_{A_{\mathsf{nd}}}(P_{*}, -)\colon A_{\mathsf{nd}}-\mod \rightarrow O_{\mathsf{nd}}-\mod$ is exact and induces an equivalence
of abelian categories $A_{\mathsf{nd}}-\mod/\add(S_{+} \oplus  S_{-}) \xrightarrow{\sim} O_{\mathsf{nd}}-\mod$ \cite[Theorem 4.8]{Tilting}. Using Miyachi's compatibility of Serre and Verdier quotients
\cite[Theorem 3.2]{Miyachi}, we see that $\PP=\mathsf{Hom}_{A_{\mathsf{nd}}}(P_{*}, -)\colon D^b\bigl(A_{\mathsf{nd}}-\mod\bigr) \rightarrow D^b\bigl(O_{\mathsf{nd}}-\mod\bigr)$ is a quotient functor in the sense of \cite{Chen}, i.e.~$D^b\bigl(A_{\mathsf{nd}}-\mod\bigr)/\ker(\PP) \xrightarrow{\sim} D^b\bigl(O_{\mathsf{nd}}-\mod\bigr)$.  A direct calculation shows that
$\mathsf{Hom}_{A_{\mathsf{nd}}}(P_{*}, P_{*}) \cong O_{\mathsf{nd}}.$
Hence, $\PP$ induces a functor $\II \colon \Delta_{\mathsf{nd}} \rightarrow \underline{\mathsf{MCM}}(O_{\mathsf{nd}})$ and the following diagram commutes.
$$
 \begin{xy}\SelectTips{cm}{}
\xymatrix{
 \Hot^b\bigl(\add(P_{*})\bigr) \ar@{^{(}->}[rr]  \ar[d]^{\PP} && D^b(A_{\mathsf{nd}}-\mod) \ar[rr]^{\mathsf{can}} \ar[d]^{\PP}   && \Delta_{\mathsf{nd}} \ar[d]^{\II}
 \\
\Perf(O_{\mathsf{nd}}) \ar@{^{(}->}[rr] && D^b(O_{\mathsf{nd}}-\mod)  \ar[rr]^{\mathsf{can}}  && \underline{\mathsf{MCM}}(O_{\mathsf{nd}})
}
\end{xy}
$$
By \cite[Lemma 2.1]{Chen}, $\II$ is again a quotient functor. Using the classification of indecomposable objects in $\Delta_{\mathsf{nd}}$ and the fact that $\mathsf{Hom}_{A_{\mathsf{nd}}}(P_{*}, P_{+} \oplus P_{-}) \cong \kk\llbracket u\rrbracket \oplus \kk\llbracket v\rrbracket,$ we obtain $\ker(\II)=\mathsf{Tria}(S_{+}, S_{-})$. This concludes the proof.\end{proof}

\begin{remark}
After submitting this article, we learned that Thanhoffer de V\"olcsey and Van den Bergh proved a general Theorem, which contains Proposition \ref{P:Quotient} as a special case. Namely, in the notations of Section \ref{S:KTheory} assume that $O$ is an \emph{isolated} singularity and $A$ has finite global dimension. Let $e\in A$ be the idempotent corresponding to the identity endomorphism of $O$. We denote the simple $A$--modules $S_{0}, \ldots, S_{r}$ in  such a way that $S_{0}$ has projective cover $Ae$. Then the exact functor $eA \otimes_{A} -\colon A-\mod \rightarrow O-\mod$ induces an equivalence of triangulated categories
\begin{align}
\frac{\Delta_{O}(A)}{\mathsf{Tria}(S_{1}, \ldots, S_{r})} \stackrel{\sim}\longrightarrow \underline{\MCM}(O),
\end{align}
see \cite[Theorem 5.1.1 and Lemma 5.1.3]{ThanhofferVandenBergh}.
Moreover, they show that $\Delta_{O}(A)$ is Hom--finite in this generality, see
\cite[Proposition 5.1.4]{ThanhofferVandenBergh}.
Using quite different methods, slight generalizations of both results were subsequently
obtained  in recent  work of Kalck and Yang \cite{KalckYang}.
\end{remark}

\section{Summary}
In this section, we collect the major results  obtained in this article.
Let $Y$ be a nodal algebraic curve, $Z$ its singular locus,  $\kI = \kI_Z$ and
$\YY = (Y, \kA)$ for $\kA = {\mathcal End}_Y(\kO \oplus \kI)$.
Similarly, let $O = \kk\llbracket u, v\rrbracket/(uv)$, $\idm = (u, v)$
 and $A = {\End}_O(O \oplus \idm)$. Then the following results are true.
\begin{itemize}
\item The category $\Delta_Y(\YY)$ splits into a union of $p$ blocks $\Delta_{\mathsf{nd}}$, where
$p$ is the number of singular points of $Y$ and  $\Delta_{\mathsf{nd}} = \Delta_O(A)$, see Corollary
\ref{C:nodalcurve}.
\item The category  $\Delta_{\mathsf{nd}}$ is $\Hom$--finite and representation discrete. In particular, its indecomposable objects and the morphism spaces between them are explicitly known, see Theorem \ref{T:MainT}. Moreover, one can compute its Auslander--Reiten quiver, see Remark \ref{r:AR-quiver}.
\item We have: $K_0\bigl(\Delta_{\mathsf{nd}}\bigr) \cong \mathbb{Z}^2$, see Theorem \ref{T:ResultsonVerdierQout}.
\end{itemize}
Moreover, the category $\Delta_{\mathsf{nd}}$ admits  an alternative ``quiver description'' in terms
of representations of a certain gentle algebra $\Lambda$, see Section \ref{S:Quivers}.

\end{document}